\documentclass[12pt,a4paper]{amsart}

\usepackage[american]{babel}
\usepackage{bm,amsmath,amssymb}
\usepackage{amsbsy,amsfonts,array}
\usepackage{verbatim,graphicx}
\usepackage{epstopdf}
\usepackage{subfigure}
\graphicspath{{./Figures/},{./}}
\usepackage{color}
\DeclareGraphicsExtensions{.jpg,.png,.pdf}
\usepackage{dsfont} 
\usepackage{amsthm}
\usepackage{textcomp}                      
\usepackage[mathscr]{euscript}

\newcommand{\Frac}[2]{\displaystyle\frac{#1}{#2}}

\newcommand{\vect}[1]{\bm{#1}}

\newcommand{\Int}[2]{\displaystyle\int_{#1}^{#2}}
\newcommand{\unit}[1]{\mathrm{#1}}

\newtheorem{remark}{Remark}[section]

\newtheorem{proposition}{Proposition}[section]
\newtheorem{theorem}{Theorem}[section]

\begin{document}

\title[3D Ion Channels]
{Three-Dimensional Simulation of Biological Ion 
Channels Under Mechanical, Thermal and Fluid Forces}

\author{Riccardo Sacco$^{1}$ \and Paolo Airoldi$^{1}$ 
\and Aurelio G. Mauri$^{1}$
\and Joseph W. Jerome$^{2}$}

\address{$^{1}$ Dipartimento di Matematica,
               Politecnico di Milano \\
	       Piazza L. da Vinci 32, 20133 Milano, Italy \\
$^{2}$ Northwestern University, Mathematics Department\\
	2033 Sheridan Road Evanston, IL 60208-2730, USA}

\date{\today}

\begin{abstract}
In this article we address the three-dimensional modeling and 
simulation
of biological ion channels using a continuum-based approach. 
Our multi-physics formulation self-consistently
combines, to the best of our knowledge for the first time, 
ion electrodiffusion, channel fluid motion, thermal self-heating
and mechanical deformation. The resulting system of nonlinearly
coupled partial differential equations in conservation form is
discretized using the Galerkin Finite Element Method.
The validation of the proposed computational model is carried out 
with the simulation of a cylindrical voltage operated ion nanochannel 
with K$^+$ and Na$^+$ ions. We first investigate the coupling between 
electrochemical and fluid-dynamical effects. Then, we 
enrich the modeling picture by investigating the influence of a thermal gradient. Finally, we add a mechanical stress 
responsible for channel deformation and investigate its effect 
on the functional response of the channel. Results show
that fluid and thermal fields have no influence 
in absence of mechanical deformation whereas ion distributions and 
channel functional response are significantly modified 
if mechanical stress is included in the model. These predictions agree with 
biophysical conjectures on the importance of protein conformation 
in the modulation of channel electrochemical properties.
\end{abstract}

\maketitle

\section{Introduction and Motivation}\label{sec:intro}
Ion channels connect the intracellular environment
and the surrounding biological medium, allowing
the dynamical exchange of the chemical species that supervise the functions 
of every cell in the human body. Ion channels control 
electrical signal transmission among excitable cells, 
notably, muscle cells and/or neurons, and because of their fundamental 
role they have been
subject to biophysical and mathematical investigation for a long time.

The most commonly used approach to ion channel modeling is based on the representation of the cell membrane lipid bilayer as an equivalent 
electrical circuit including capacitive and conductive elements, 
these latter being in general described by a nonlinear current-voltage 
characteristic. The resulting formulation is constituted by 
a nonlinear system of ordinary differential equations (ODEs) 
whose best known example is the Hodgkin-Huxley (HH)
model~\cite{HH1952a,HH1952b}. We refer 
to~\cite{KeenerSneyd1998,ErmentroutBook} for a detailed treatment
of ODE-based differential models and their application 
to relevant problems in Computational Biology.

ODE differential models are widely used because of their
limited computational cost and ease of implementation. 
However, their biophysical accuracy is limited and often allows to 
characterize only basic properties of the cellular system under 
investigation, such as the homeostatic Nernst potential~\cite{Hille2001}, 
and/or to reproduce simple experiments in Electrophysiology, such as the 
transmembrane current in voltage clamp conditions
~\cite{HodgkinHuxleyKatz1952c,HodgkinHuxley1952d}.

To analyze the fundamental mechanisms that govern ion transport 
across a membrane channel, a higher model complexity is required.
With this aim, a system of partial differential equations (PDEs) 
expressing the balance of mass and of linear momentum 
for each ion species can be adopted to describe the dynamical balance 
of electro-diffusion forces acting on chemicals flowing across the
cell membrane bilayer. The resulting formulation is represented 
by the Poisson-Nernst-Planck (PNP) model for ion electro-diffusion.
We refer to~\cite{BarcilonPartI1992,BarcilonPartII1992,park,Barcilon},
for a detailed illustration of the PNP differential system and 
a discussion of its mathematical and numerical properties.
\begin{figure}[h!]
\centering
\includegraphics[width=0.6\columnwidth]{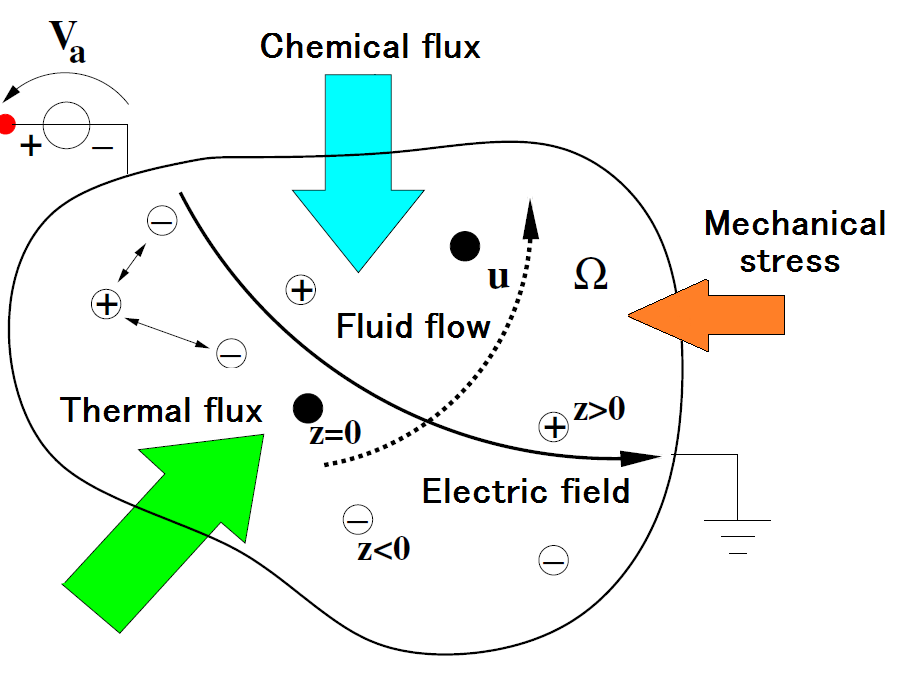}
\caption{Scheme of the general physical problem.}\label{fig:scheme_problem}
\end{figure} 

In the present work we address the three-dimensional (3D) 
study of a generalized version of the PNP model,  
modified by including the effect of electrolyte fluid convective 
transport, a thermal gradient and mechanical forces. 
The extension of the PNP system accounting for fluid velocity
is well-known as the {\it velocity-extended} PNP model and has been
the object of extensive mathematical and numerical investigation 
in~\cite{Rubinstein1990,Jerome2002,Sacco2009,Schmuck,Mauri2014}.
In these pages we adopt a wider vision of ion transport as 
schematically depicted in Fig.~\ref{fig:scheme_problem} 
where $\Omega$ is an open bounded domain representing 
the bio-physical system that includes the ions in motion 
through a physiological fluid
under the action of several externally applied forces: electrical, thermal, chemical and mechanical. Each of these external forces 
contributes to determining ion flow in the medium according to the effect of
the following vector fields:
\begin{enumerate}
\item electric field;
\item concentration gradient;
\item thermal gradient;
\item electrolyte fluid.
\end{enumerate}
In typical applications of Electrophysiology and Cellular Biology
the electric field may be the result of the application to the cell of an external voltage drop such as in a patch clamp 
experiment~\cite{neher1992patch,Levis_Rae,huxley2002overshoot},
and/or the presence of fixed charged ions embedded in the portion of 
lipid bilayer surrounding a transmembrane channel~\cite{Hille2001,Siwy14}.
The ion concentration gradient may be the result of the administration of a substance such as a drug or a toxin~\cite{Friel1992,Straub2001,hutzler2006high,zeck2001noninvasive}, 
whereas the thermal gradient may be the result of the activation 
of heat-sensitive thermo-transient receptor potential (thermo-TRP) 
cation channels exposed to capsaicin, a natural ingredient of spicy foods such as red hot chilli peppers and of cold-sensitive TRP channels exposed to 
menthol~\cite{caterina1997,voets2004}. 
Ion motion can also be activated by a physiological 
fluid shear stress acting on a mechanosensitive 
ion channel~\cite{Sachs2010}. Conversely, ion motion can
determine a modification of the physical properties of the surrounding 
environmnent through the effect of electro-osmotic 
pressure~\cite{Kedem1958,ChangYaoBook}, whereas the mechanical 
action can be determined by the chemicals signal acting to change 
the size of the channel or of the vessel in which the fluid flows
as in the mechanism of neurovascular 
coupling~\cite{attwellglial2010,Newman2013}.

The 3D simulation of the Generalized PNP equations 
(for short, GenPNP) including 
electrolyte fluid convective transport, thermal gradient 
and mechanical forces is a formidable
task. The presently available computer resources and 
numerical discretization schemes, the Galerkin Finite Element Method
(GFEM) in our case, open the possibility 
to explore at the cellular scale, 
to the best of our knowledge for the first time, 
the multi-physics complexity of the phenomena 
illustrated above. The computational structure devised 
to numerically solve in a fully self-consistent manner the GenPNP system 
is part of the software MP-FEMOS 
(Multi-Physics Finite Element Modeling Oriented Simulator) 
that has been developed by some of the authors~\cite{Mauri2014,Mauri2014currents,Airoldi2015}. 
MP-FEMOS is a general-purpose modular 
numerical code based on the use of the 
GFEM implemented in a fully 3D framework through 
shared libraries using an objected-oriented programming language (C++).  
Several kinds of finite elements on simplexes are included for 
appropriate discretization of the various PDEs constituting the 
problem at hand. In particular, the electro-chemical and thermal 
blocks are solved using the exponentially fitted edge averaged
(EAFE) piecewise linear finite element scheme studied in~\cite{gatti1998new} 
and~\cite{Zikatanov:EAFE}; the flow of the incompressible electrolyte fluid
is solved using the Taylor-Hood finite element pair 
(see~\cite{QV}, Chapter 9), whereas the mechanical deformation of the 
channel is numerically computed using standard 
piecewise linear finite elements.

An outline of the article is as follows. 
In sec.~\ref{sec:geometry_fem} we introduce the basic notation of 
the problem at hand, its geometry and partition into simplices.
In sec.~\ref{sec:model} the mathematical model of the GenPNP system
is discussed in detail. In sec.~\ref{sec:algorithms} 
the solution algorithm employed to decouple the various blocks of 
the full system is described. Sec.~\ref{sec:weak} is 
devoted to illustrate the weak formulation for each differential subproblem whereas sec.~\ref{sec:feapproximation} treats the adopted 
finite element discretization schemes. 
Sec.~\ref{sec:simulations} contains an extensive bouquet of 
simulations aimed at validating the biophysical accuracy of 
model and computational schemes. Finally, 
sec.~\ref{sec:conclusions} collects concluding remarks and future
research objectives.

\section{Basic Notation, Geometry and Geometrical Discretization}\label{sec:geometry_fem}

Referring to Fig.~\ref{fig:scheme_problem}, 
in this section we introduce the basic nomenclature associated with the
multi-physics problem, the computational domain in which the problem 
has to be solved and its geometrical discretization needed for the subsequent
numerical approximation with the GFEM. 

\subsection{Geometry and notation}
Let $\mathbf{x}$ and $t$ denote the spatial and temporal coordinates,
respectively. The simulation domain $\Omega \subset \mathbb{R}^3$
is assumed to be an incompressible, viscous and 
Newtonian fluid moving with velocity
$\mathbf{u}=\mathbf{u}(\mathbf{x},t)$, 
in which $M \geq 1$ chemical species are dissolved. Each chemical has 
effective charge $z_i$, $i=1,...,M$, and concentration (number density)
$n_i = n_i(\mathbf{x}, t)$ [$\unit{m^{-3}}$]. Cations have $z_i >0$, anions
have $z_i <0$ while neutral species have $z_i=0$.

We denote by $\partial \Omega$ the boundary of the domain $\Omega$
and by $\mathbf{n}$ the unit outward normal vector on $\partial \Omega$.
Let $\mathcal{U}=\{ n_i, \varphi, \mathbf{u}, T, \mathbf{d} \}$
denote the set of dependent variables of the problem, 
$\varphi= \varphi(\mathbf{x}, t)$, $T = T(\mathbf{x}, t)$
and $\mathbf{d}=\mathbf{d}(\mathbf{x}, t)$ representing the 
electric potential, temperature and mechanical 
displacement field solutions of the biophysical system under investigation.
Then, with each variable $s \in \mathcal{U}$, we associate a partition
of the domain boundary $\partial \Omega$ into the union of 
(generally different) subsets. 
We indicate by $\Gamma_D^s$ the subset of $\partial \Omega$ where 
a Dirichlet condition 
and by $\Gamma_N^s$ the subset where a Neumann condition is applied, 
in such a way that $\Gamma_D^s \cup \Gamma_N^s = \partial \Omega$ 
and $\Gamma_D^s \cap \Gamma_N^s = \varnothing$. 

\subsection{Geometrical discretization of the computational domain}
In view of the numerical approximation of a given differential problem,
the MP-FEMOS computational platform provides a
flexible Galerkin Finite Element discretization 
on a fully unstructured triangulation $\mathcal{T}_h$ 
of the domain $\Omega \subset \mathbb{R}^3$ 
into tetrahedral elements $K$, such that 
$\underset{K \in \mathcal{T}_h}{\bigcup} \overline{K} = \overline{\Omega}$. 

For all $K \in \mathcal{T}_h$, we let 
$h_K = {\rm diam}(K)$ be the diameter of $K$, defined as the 
longest edge of $K$, and indicate by 
$h = \displaystyle \max_{K \in \mathcal{T}_h} \ h_K$ 
the discretization parameter. 
We assume $\mathcal{T}_h$ to be a regular partition (cf.~\cite{QV}, Chapter 
3), i.e., that there exists a positive constant $\xi$ 
(independent of $h$) such that 
\begin{equation}\label{eq:regularity_Th}
{{h_K}\over{\rho_K}} \leq \xi \qquad \forall K \in \mathcal{T}_h
\end{equation}
$\rho_K$ being the diameter of the largest sphere inscribed in $K$.
Condition~\eqref{eq:regularity_Th} is equivalent to assuming a bound from
below of the minimum mesh angle. 

\section{Mathematical Model}\label{sec:model} 
In this section we illustrate the multi-physics model for
ion electrodiffusion in the biological environment schematically
represented in Fig.~\ref{fig:scheme_problem}.

Ion and fluid motion are mathematically 
described by a coupled systems of PDEs 
including: the velocity-extended Poisson-Nernst-Planck (VE-PNP) 
equation system to account for electro-chemical forces  
and the Stokes equation to account for fluid forces
under the assumption that nonlinear convective effects can be neglected
because of slow fluid motion.

Thermal dispersion into the channel is considered by a heat 
conduction equation that allows to determine
temperature evolution at all $\mathbf{x} \in \Omega$.

To calculate the deformation of the channel wall, 
the domain $\Omega$ is regarded as a compressible linear elastic medium undergoing the classical Navier-Lam\'e theory. At present, 
no direct coupling is activated between stress and deformation fields
and the PNP system. This modeling limitation will be removed 
in a future step of our research.

\subsection{The VE-PNP Poisson-Nernst-Planck system}\label{sec:VE-PNP}

A detailed derivation of the PNP systems can be found, e.g., in~\cite{Mauri2014}. 
By the application of: 1) mass balance for each chemical; 
2) linear momentum balance (second law of dynamics); 
3) Stokes' law for viscous drag exerted by a fluid on a 
particle in motion through it; 4) Einstein's relation; 
and 5) the law of perfect gases, the following system of
PDEs is obtained: \\
$\forall i=1,...,M$, solve:
\begin{subequations}\label{eq:pnpprob}
\begin{align}
& {{\partial n_i} \over{\partial t}} + 
{\rm div \ } \mathbf{f}_i = 0 & \label{eq:massPNP} \\
& \mathbf{f_i} 
= \Frac{z_i}{|z_i|} \mu_i 
n_i \mathbf{E} - D_i \nabla n_i 
- D_i n_i \displaystyle {{\nabla T} \over{T}} 
+ \boxed{n_i \mathbf{u}} & 
\label{eq:momentumPNP} \\
& {\rm div \ } (\epsilon_f \mathbf{E}) = q \sum_{i=1}^M z_i n_i +
q \rho_{fixed} & 
\label{eq:PoissonPNP}\\
& \mathbf{E} = - \nabla \varphi. & \label{eq:efieldPNP}
\end{align}
\end{subequations}
In~\eqref{eq:pnpprob}, the symbol $\mathbf{f}_i$ denotes the 
ion particle flux [$\unit{m^{-2} s^{-1}}$], $\mu_i$ and $D_i$
are the ion electrical mobility and diffusivity and $T$ is system
temperature. In the Poisson equation~\eqref{eq:PoissonPNP}, $\mathbf{E}$ and 
$\varphi$ are the electric field [$\unit{V m^{-1}}$]
and electric potential [$\unit{V}$], respectively,
while $q$ and $\epsilon_f$ denote the electron charge and the
electrolyte fluid dielectric permittivity, respectively.
The quantity $q \rho_{fixed}$ [$\unit{C m^{-3}}$] mathematically accounts for 
the presence of a fixed charge density due to the lipid membrane bilayer 
and is assumed to be a given function of position only.
The diffusion coefficient $D_i$ and the mobility $\mu_i$ are 
proportional through the generalized (because $T$ is \emph{not}
constant) Einstein's relation
\begin{equation}\label{eq:einstein}
D_i = \mu_i {{k_B T}\over{q |z_i|}}
\end{equation}
where $k_B$ is Boltzmann's constant.
Initial conditions for ion concentrations are
$\forall i=1,...,M$:
\begin{subequations}\label{eq:pnpIC}
\begin{align}
&n_i(\mathbf{x},0) = n_i^0(\mathbf{x}) & \ \ \ \mbox{in } \ \Omega &\label{eq:PNP_initial}
\end{align}
\end{subequations}
where the functions $n_i^0$ are positive given data.
The boundary conditions for the VE-PNP system are
$\forall i=1,...,M$:
\begin{subequations}\label{eq:pnpBC}
\begin{align}
&\varphi = \overline{\varphi} &\ \ \ \mbox{on } \ \Gamma_D^{\varphi} &\label{eq:firstPNP}\\ 
&\mathbf{E} \cdot \mathbf{n} = 0 & \ \ \ \mbox{on } \ \Gamma_N^{\varphi}&\label{eq:secondPNP} \\ 
&n_i = \overline{n}_i & \ \ \ \mbox{on } \ \Gamma_D^{n_i} 
&\label{eq:thirdPNP}\\ 
&\mathbf{f_i} \cdot \mathbf{n} = 0 & \ \ \  \mbox{on }
\ \Gamma_N^{n_i} &\label{eq:fourthPNP}
\end{align}
\end{subequations}
where $\overline{\varphi}$ is the electrostatic potential of the 
side $\Gamma_D^{\varphi}$ and $ \overline{n}_i$ is a 
given concentration of the chemical species $i$ on side 
$\Gamma_D^{n_i}$.
Conditions~\eqref{eq:firstPNP} and~\eqref{eq:thirdPNP} enforce respectively 
a given voltage and a given concentration on the side, while~\eqref{eq:secondPNP} 
and~\eqref{eq:fourthPNP} enforce a homogeneous Neumann condition, 
that corresponds for chemical particles to the conservation of mass 
concentration inside the simulation domain 
(i.e. particles cannot leave the domain).

The VE-PNP equation system can be regarded as a generalization of 
the classic Drift-Diffusion (DD) model for semiconductor devices~\cite{selberherr,markowich1986stationary,Markowich,JeromeBook} 
through the addition of the thermomigration and fluid velocity terms
in the linear momentum balance equation~\eqref{eq:momentumPNP}.
As a matter of fact, the third term in~\eqref{eq:momentumPNP}
is related to the thermal grandient effect on ion
flow while the last term gives rise to an additional
translational contribution due to fluid motion
and is highlighted by a box. This term 
represents the coupling between electrolyte fluid motion
and electrodiffusive ion transport. Henceforth, we use the reference
"PNP system" every time the contribution of the velocity and of the
thermal gradient is neglected in the equation set~\eqref{eq:pnpprob}.

\subsection{The Stokes system}\label{sec:Stokes}

The Stokes system to describe the slow motion of an incompressible 
and viscous fluid with a constant density $\rho_f$  
[$\unit{Kg \ m^{-3}}$] consists of the following PDEs (see~\cite{QV}
for a complete mathematical and numerical treatment):
\begin{subequations}\label{eq:stpr}
\begin{align}
& { \rm div \ } \mathbf{u} = 0 & \label{eq:contStokes} \\
& 
\rho_f \displaystyle {{\partial \mathbf{u}}\over{\partial t}} =
{ \rm div \ } \underline{\underline{\sigma}}(\mathbf{u},p) +
\boxed{q \sum_{i=1}^M z_i n_i \mathbf{E}} & \label{eq:momentumStokes} \\
& \underline{\underline{\sigma}}(\mathbf{u},p) = 2 \mu_f \underline{\underline{\epsilon}}(\mathbf{u}) - 
p \underline{\underline{ \delta}} & \label{eq:stressStokes} \\
& \underline{\underline{\epsilon}}({\mathbf{u}}) = 
\underline{\underline{\nabla}}_{\, s}({\mathbf{u}}) =
{{1}\over{2}} (\nabla \mathbf{u} + (\nabla \mathbf{u})^T) & \label{eq:strainStokes}  
\end{align}
\end{subequations}
where $\mathbf{u}$ is the fluid velocity [$\unit{m s^{-1}}$], 
$p$ is the fluid pressure [$\unit{Pa}$], 
$\mu_f$ the fluid shear viscosity [$\unit{Kg \ m^{-1} s^{-1}}$], $\underline{\underline{\sigma}}$ 
the stress tensor [$\unit{Pa}$] and $\underline{\underline{\epsilon}}$ 
the strain rate tensor [$\unit{s^{-1}}$].
The symbol $\underline{\underline{ \delta}}$ is the second-order 
identity tensor of dimension 3 whereas the second-order tensor 
$\underline{\underline{\nabla}}_{\, s}
({\mathbf{u}})$ is the symmetric gradient of $\mathbf{u}$.
Notice that in accordance with the assumption of slow 
fluid motion, the quadratic convective term in the inertial forces has 
been neglected in the momentum balance equation~\eqref{eq:momentumStokes}.
The boxed term at the right-hand side of equation~\eqref{eq:momentumStokes}
physically corresponds to the electric pressure exerted 
by the ionic charge on the electrolyte fluid, and mathematically
represents the coupling between electrodiffusive ion transport and 
electrolyte fluid motion.
The initial condition for electrolyte fluid velocity is
\begin{subequations}\label{eq:stIC}
\begin{align}
&\mathbf{u}(\mathbf{x},0) = \mathbf{u}^0(\mathbf{x}) 
& \ \ \ \mbox{in } \ \Omega &\label{eq:st_initial}
\end{align}
\end{subequations}
where the function $\mathbf{u}^0$ is a given datum, usually set equal to
zero. The boundary conditions for the Stokes system are:
\begin{subequations}\label{eq:stBC}
\begin{align}
&\mathbf{u} = \mathbf{g} & \ \ \ \mbox{on } \ \Gamma_D^{\mathbf{u}} &\label{eq:firstST}\\ 
&\underline{\underline{\sigma}}(\mathbf{u}, p)\mathbf{n} 
= \mathbf{h} & \ \ \ \mbox{on } \ \Gamma_N^{\mathbf{u}}&\label{eq:secondST}
\end{align}
\end{subequations}
where $\mathbf{g}$ is the imposed velocity on the side 
$\Gamma_D^{\mathbf{u}}$ and $\mathbf{h}=\overline{p}\mathbf{n}$ 
is the value of the normal stress on $\Gamma_N^{\mathbf{u}}$.
Relation~\eqref{eq:firstST} is the wall adherence condition 
(typically $\mathbf{g} = \mathbf{0}$) 
and~\eqref{eq:secondST} is the application of the 
action-reaction principle for the normal 
stress $\underline{\underline{\sigma}}\mathbf{n}$.
The presence of the two boxed terms  
in the linearized momentum balance equations~\eqref{eq:momentumPNP} 
and~\eqref{eq:momentumStokes}
terms requires the adoption of a suitable iterative procedure
to successively solve the whole coupled 
system~\eqref{eq:pnpprob}-~\eqref{eq:stpr} 
as described in sec.~\ref{sec:algorithms}.

\subsection{Heat equation}

Ion and fluid motion, as well as the action of mechanical forces,
may dissipate energy throughout the channel. 
The evolution of the resulting channel thermal profile 
$T = T(\mathbf{x},t)$ in $\Omega$ is obtained
by solving the thermal conduction equation (see~\cite{landaufluid}) 
given by
\begin{subequations}\label{eq:Tpr}
\begin{align}
& {{\partial (\rho_f c T)}\over{\partial t}} + 
{ \rm div \ } \mathbf{q} = Q_{heat}  & \label{eq:Teq} \\
& \mathbf{q} = - k \nabla T & \label{eq:Tfluxeq} 
\end{align}
\end{subequations}
where $c$ is the specific heat [$m^2 s^{-2} K^{-1}$], 
$\mathbf{q}$ is the heat flux [$W m^{-2}$], $k$ is 
the thermal 
conductivity [$W m^{-1} K^{-1}$] and $Q_{heat}$ is 
the heat generation term [$W m^{-3}$]. We notice that
convective thermal effects due to 
fluid velocity are neglected in the mathematical definition of the heat flux 
$\mathbf{q}$. Also, we notice that the heat production term $Q_{heat}$ is the 
result of heat source/sink processes such as Joule heating or 
(possibly nonlinear) thermo-chemical reactions due to 
the administration of drugs as in TRP channels.
For the purpose of the present work, $Q_{heat}$ is set equal to 
zero, but this limitation will be removed in a subsequent step
of the research.
The initial condition for electrolyte fluid temperature is
\begin{subequations}\label{eq:TIC}
\begin{align}
&T(\mathbf{x},0) = T^0(\mathbf{x}) 
& \ \ \ \mbox{in } \ \Omega &\label{eq:T_initial}
\end{align}
\end{subequations}
where the function $T^0$ is a positive given datum. 
The boundary conditions for the channel heat equation are:
\begin{subequations}\label{eq:bcC}
\begin{align}
&T = \overline{T}& \ \ \ \mbox{on } \ \Gamma_D^{T} &\label{eq:firstTE}\\ 
&\mathbf{q} \cdot \mathbf{n} = 0 & \ \ \ 
\mbox{on }\ \Gamma_N^{T}&\label{eq:secondTE}
\end{align}
\end{subequations}
where $\overline{T}$ is the imposed temperature on the side 
$\Gamma_D^{T}$.
Relation~\eqref{eq:firstTE} represents 
the effect of an infinite heat source at temperature 
$T=\overline{T}$ while~\eqref{eq:secondTE} is the adiabatic condition.

\subsection{The Navier-Lam\'e approach to mechanical equilibrium}

The mechanical problem of channel deformation is here described 
by the model of linear isotropic elasticity 
including thermal expansion and represented by the following system
expressing the equilibrium condition and 
the generalized Hooke's law:
\begin{subequations} \label{eq:MECpr}
\begin{align}
& {\rm div \ } \vect{\underline{\underline{\sigma}}}_{mec} (\vect{d}) + \vect{f_{mec}}= 0  &  \label{eq:MECeq} \\ 
&\vect{\underline{\underline{\sigma}}}_{mec} (\vect{d};T) = \underline{\underline{\textit{C}}} (\underline{\underline{\varepsilon}}_{mec} (\vect{d}) - \underline{\underline{\varepsilon}}_{mec}^{th} (T))+ \underline{\underline{\sigma_0}} & \label{eq:MECstress} \\
&\underline{\underline{\varepsilon}}_{mec}:= \nabla_s(\vect{d}) & 
\label{eq:MECdef}
\end{align}
\end{subequations}
where $\vect{d}=[d_x,d_y,d_z]^T$ is the displacement [$m$], $\underline{\underline{\sigma}}_{mec}$ is
the stress tensor [$Pa$], $\vect{f}_{mec}$ is the volumetric 
force density [$N/m^3$], 
$\underline{\underline{\varepsilon}}_{mec}= \underline{\underline{\varepsilon}}_{mec}^{th} + \underline{\underline{\varepsilon}}_{mec}^{el}$ is 
the total strain tensor, $\underline{\underline{\varepsilon}}_{mec}^{el}$ 
is the elastic strain tensor, $\underline{\underline{\varepsilon}}_{mec}^{th}$ 
is the thermal strain tensor, $\underline{\underline{\textit{C}}}$ 
is the elastic tensor [$Pa$] and $\underline{\underline{\sigma_0}}$ is the initial stress [$Pa$].
In particular, if the thermal properties are homogeneous and orthotropic the thermal deformation
(strain) is given by the relation
\begin{align*}
\underline{\underline{\varepsilon}}_{mec,ii}^{th} = \alpha_i \cdot(T -T_{ref})
\end{align*}
where $\alpha_1$, $\alpha_2$ and $\alpha_3$ [$K^{-1}$] are the linear 
thermal expansion coefficients, and $T_{ref}$ is the reference 
temperature [$K$]. Under standard assumptions (symmetric 
stress and strain, symmetric and positive definite 
energy functional and coordinate invariance of mechanical response) 
the relation between stress and elastic strain becomes
\begin{align}
\underline{\underline{\sigma}}_{mec,ij} = 2\mu \varepsilon^{el}_{mec,ij} + \lambda \varepsilon^{el}_{mec,kk} \quad i,j=1,2,3 & \label{eq:Hooke} 
\end{align}
where 
$$
\lambda= \frac{\nu E}{(1+\nu)(1-2\nu)}\qquad \quad
\mu= \frac{E}{2(1+\nu)} 
$$
are the Lam\'e coefficients, while $0<\nu< 0.5$ is the Poisson ratio 
and $E > 0$ is the Young module [$Pa$] describing the material mechanical properties. 
In the mechanical system, at each time level $t$, 
the boundary conditions are expressed by a given displacement 
on side $\Gamma_D^{\mathbf{d}}$ or by a stress free surface on side $\Gamma_N^{\mathbf{d}}$:
\begin{subequations}\label{eq:bcs_mech}
\begin{align}
&\mathbf{d} = \mathbf{d}_D & \ \ \ \mbox{on } \ \Gamma_D^{\mathbf{d}} &\label{eq:firstNL}\\ 
&\underline{\underline{\sigma}}(\mathbf{d}; T)\mathbf{n} = 0 & \ \ \ 
\mbox{on } \ \Gamma_N^{\mathbf{d}}&\label{eq:secondNL}
\end{align}
\end{subequations}
where $\mathbf{d}_D$ is a given boundary displacement.

\subsection{Fully coupled multi-physical approach}

The multi-physics description of biological ion nanochannels is hence
obtained by combining together all the previous blocks as follows:
\begin{itemize}
\item Mechanical description ~\eqref{eq:MECeq}
\item Ionic transport~\eqref{eq:pnpprob}
\item Thermal description~\eqref{eq:Teq}
\item Fluid velocity calculation~\eqref{eq:stpr}
\end{itemize}
We define this as the the Fully Coupled Thermal Velocity Extended 
Poisson Nernst Planck Stokes system
in presence of mechanical forces or, shortly, FC-MF-T-VE-PNP-S model.
To the best of our knowledge, such a formulation represents the first 
example of a multi-physics based formulation for ion channel modeling
and simulation. 

\section{Solution Algorithm}\label{sec:algorithms}

In view of the numerical solution of the FC-MF-T-VE-PNP-S system,
for a given time $T_f>0$ (final simulation time), we introduce the 
time interval $I_{T_f}:= [0, T_f]$ and we divide $I_{T_f}$ into 
a uniform partition of $N_{T_f} \geq 1$ time slabs 
$\tau^k = (t^{k-1}, t^{k})$ of width $\Delta t = T_f/N_{T_f}$, 
such that $\cup_k \tau^k = I_{T_f}$. Then, for each $k =0, 
\ldots, N_{T_f}$, we denote by: 
\begin{subequations}
\begin{align}
& \mathbf{U}_{MEC} = \left[
d_x, d_y, d_z \right] & 
\label{eq:solutionvectorMEC} \\
& \mathbf{U}^{k}_{PNP} = \left[
n_1^{k}, n_i^{k}, \ldots, n_M^{k}, \varphi^{k}, T^k\right] & 
\label{eq:solutionvectorPNP} \\
& \mathbf{U}^{k}_{Stokes} = \left[
\mathbf{u}^{k}, p^{k}\right] & 
\label{eq:solutionvectorStokes} \\
& \mathbf{U}^{k} = \left[ \mathbf{U}^{k}_{PNP}, \mathbf{U}^{k}_{Stokes} 
\right]^T & \label{eq:solutionvector}
\end{align}
\end{subequations}
the solution vectors at time level $t^k$ 
of the Mechanical block and those of the 
Thermal PNP block, of the Stokes block
and of the whole coupled problem. 
The solution of the FC-MF-T-VE-PNP-S system
using a monolithic scheme is a formidable computational task.
Thus, the use of a decoupled approach is highly desirable.
For each discrete time level $t^k$, $k=1, \ldots, N_{T_f}$, 
a two-subblock iteration is carried out. 
In the first step of the iteration
the mechanical problem is solved for the displacement 
$\mathbf{U}_{MEC}$. In the second step of the 
iteration $\mathbf{U}_{MEC}$ 
is used to compute the corresponding deformation 
$\theta = \theta(\mathbf{U}_{MEC})$ providing a new domain configuration
in which the remaining thermal, electrochemical and fluid 
blocks of the problem are successively solved. In doing this, 
temporal semidiscretization is performed using different temporal schemes: 
a) one-step methods: Backward Euler (BE) and
Trapezoidal Rule (TR); b) two-step method: TR-BDF2. We refer
to~\cite{quarteroni2007numerical}, Chapter 11, 
for a detailed discussion and analysis of these methods.

The flow-chart of the 
iteration map that is implemented to solve successively each block
of the FC-MF-T-VE-PNP-S system is illustrated in 
Fig.~\ref{fig:gummelchart} where  
$k = 0, \ldots N_{T_f}-1$ is the temporal loop counter 
and $j$ is the iteration counter in the PNP cycle.
In detail, the successive solution of the T-VE-PNP-S system
proceeds as follows.
\begin{description}
\item[Step 1:] solve successively 
the PNP system~\eqref{eq:pnpprob} and~\eqref{eq:Teq}
using $\mathbf{u}=\mathbf{u}^k$, until self-consistency is obtained 
among
$n_i$, $\varphi$ and $T$. This step is the Thermal-PNP cycle and returns 
$\mathbf{U}^{k+1}_{PNP}$.
\item[Step 2:] solve the Stokes system~\eqref{eq:stpr} with
$\mathbf{E}= -\nabla \varphi^{k+1}$ using the Uzawa iterative 
algorithm (see~\cite{QV}, Chapter 9) or a direct method. 
This step is the Stokes cycle
and returns $\mathbf{U}^{k+1}_{Stokes}$.
\end{description}

\begin{figure}[htbp]
\centering
\includegraphics[width=0.5\linewidth]{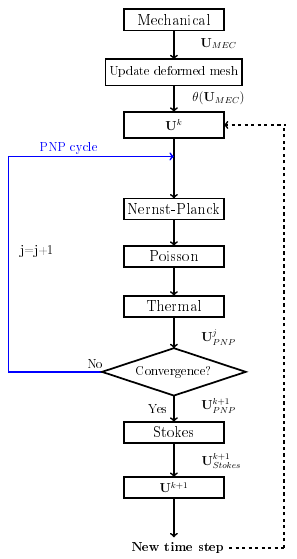}
\caption{Flow chart of the solution algorithm.}
\label{fig:gummelchart}
\end{figure}
The criterion adopted to monitor the convergence of the
solution map of Fig.~\ref{fig:gummelchart} is to stop the algorithm 
at the first value $j^{*} > 0$ of the iteration counter $j$ such that
\begin{equation}
\left\| U^{(j^{*}+1)}_{PNP}-U^{(j^{*})}_{PNP} \right\|_2 < toll
\label{eq:norm1}
\end{equation}
where $toll$ is a prescribed tolerance, and
\begin{equation}
\left\| \textbf{w}\right\|_2 = \left(\sum^p_{i=1} w_i^2 \right)^{1/2}
\label{eq:norm2}
\end{equation}
is the $2-$norm of a vector $\textbf{w} \in \mathbb{R}^p$. 
In the numerical experiments we set $toll=10^{-3}$.
The above described solution map is a generalization of the
classic Gummel decoupled algorithm widely used in contemporary 
semiconductor device simulation tools for the DD model, 
and analyzed in detail
in the stationary case in~\cite{markowich1986stationary,JeromeBook} 
and in~\cite{Markowich} in the time-dependent case.
The convergence rate of the generalized Gummel algorithm is linear;
however, there are two main advantages by using such an approach 
instead of applying, for instance, a fully coupled Newton Method.
A first advantage is that in the case of the Gummel method 
memory cost is strongly reduced so that
also the overall computational effort is much less intense. 
A second advantage is that the
Gummel method is more insensitive to the choice of the initial guess
than Newton's iteration. This fact is particularly relevant 
in multi-dimensional (and, especially, multi-physics) 
problems, like that under 
consideration, where the construction of an appropriate 
initial guess is in general not an easy task, sometimes even impossible. 

\section{Weak formulation and well-posedeness analysis}\label{sec:weak}

In this section we discuss the weak formulation 
and the well-posedeness of each differential subblock 
in the solution algorithm of sec.~\ref{sec:algorithms}.
In sec.~\ref{TVEPNP_well} we study the T-VE-PNP system,
the Stokes system in sec.~\ref{Stokes_well}, 
and the Navier Lam\'e formulation for the mechanical displacement
in sec.~\ref{Mech_well}.
The analysis of the Poisson equation~\eqref{eq:PoissonPNP} and heat equation~\eqref{eq:Teq} is standard and can be found, e.g., 
in~\cite{hughesbook,QV}.
For ease of notation, we set $\xi^k:= \xi(\mathbf{x}, t^k)$, 
$k = 0, \ldots, N_{T_f}$, for any function $\xi = \xi(\mathbf{x}, t)$.
To avoid technicalities in the subsequent analysis we assume, without
loss of generality, that the Dirichlet data $\overline{n}_i$ 
(T-VE-PNP system), $\mathbf{g}$ (Stokes system) and $\mathbf{d}_D$ (mechanical system) are equal to zero. Moreover, we restrict the study to 
the case where the BE method 
is used for time discretization as the analysis proceeds in a similar manner
if the TR or the TR-BDF2 scheme is employed.

\subsection{T-VE-PNP system}\label{TVEPNP_well}

After time discretization, by using the Einstein relation~\eqref{eq:einstein}
in~\eqref{eq:momentumPNP} and the boundary conditions
~\eqref{eq:massPNP} and~\eqref{eq:momentumPNP},
the T-VE-PNP system takes the following form:
$\forall k = 0, \ldots, N-1$ and for all $j \geq 0$ until convergence
of the PNP cycle of Fig.~\ref{fig:gummelchart}, solve:
\begin{subequations}\label{eq:pforte}
\begin{align}
& \displaystyle {{ n_i^{j+1} - n_i^k}\over{\Delta t}} 
+ { \rm div \ } \mathbf{f}_{n_i}^{j+1} = 0 & \mbox{in } \ \Omega & \label{eq:pforte_cont} \\
& \mathbf{f}_{n_i}^{j+1} = -{D}_i^{j} 
\left(\nabla n_i^{j+1} + 
\displaystyle{{q z_i}\over{k_B T^{j}}} n_i^{j+1}
\nabla \varphi^{j+1}
+ n_i^{j+1} \displaystyle {{\nabla T^{j}}\over{T^{j}}}\right) + 
{n}_i^{j+1} \mathbf{u}^k & \mbox{in } \ \Omega & \label{eq:pforte_flux} \\
& \mathbf{f}_{n_i}^{j+1} \cdot \mathbf{n} = 0	& 
\mbox{on } \ \Gamma_N^{n_i} & \label{eq:pforte_bc_N} \\
& n_i^{j+1} = 0 & \mbox{on } \ \Gamma_D^{n_i} &
\label{eq:pforte_bc_D}
\end{align}
\end{subequations}
where $\varphi^{j+1}$ is a given function 
computed by solving the Poisson equation~\eqref{eq:PoissonPNP}, 
$T^{j}$ is the solution of the heat equation~\eqref{eq:Tpr} and
$D^j_i$ is the diffusion coefficient of species $i$ computed using~\eqref{eq:einstein} with $T = T^{j}$.

The equation pair~\eqref{eq:pforte_cont}-\eqref{eq:pforte_flux}
is of the admissible form (8.1) of~\cite{gilbarg1998elliptic}.
However, it is immediate to check that the second a priori
estimate in formula (8.6) of~\cite{gilbarg1998elliptic} is, in general, 
not satisfied preventing a theoretical study of the well posedness of the equation system~\eqref{eq:pforte}. To overcome this impasse, 
and also in view of the finite element discretization 
of~\eqref{eq:pforte}, we replace
the expression~\eqref{eq:pforte_flux} with a suitable approximation.
To this purpose, 
for any function $q : \mathcal{T}_h \rightarrow \mathbb{R}$
such that $q \in W^1(K)$ for all $K \in \mathcal{T}_h$, we denote the restriction of $q$ to an element $K \in \mathcal{T}_h$ by $q|_K$ and define 
the piecewise gradient operator $\nabla_{\mathcal{T}_h} q : \mathcal{T}_h \rightarrow (\mathbb{R})^3$ as the vector-valued function such that
$$
\left(\nabla_{\mathcal{T}_h} q\right)|_K:= 
\nabla \left(q|_K\right) \qquad \forall K \in \mathcal{T}_h.
$$
Finally, we associate with the function $q|_K$ its {\emph harmonic
average} $H_K(q)$ defined as
$$
H_K(q) : = \left( \Frac{\int_K q^{-1} \, dK}{|K|} \right)^{-1}
\qquad \forall K \in \mathcal{T}_h.
$$
Then, the approximate form of the flux~\eqref{eq:pforte_flux} that is
used in this article is
\begin{equation}\label{eq:approx_flux_pnp}
\mathbf{f}_{n_i}^{k+1} = -{D}_i^{j} 
\displaystyle  \left(\nabla n_i^{j+1} + 
\Frac{q z_i}{k_B H_{\mathcal{T}_h}(T^{j})} 
n_i^{j+1}\nabla \varphi^{j+1} + n_i^{j+1} 
\Frac{\nabla T^{j}}{T^{j}} \right) + n_i^{j+1} \mathbf{u}^k.
\end{equation}
The above definition corresponds to having replaced in the electrical 
drift term of the flux the local environmental temperature with
its harmonic average. It is well known (see~\cite{BabuskaMixMet})
that the use of the harmonic average for 
rapidly varying functions provides much more accurate results than
the usual integral average. 
It is expected that as the mesh size $h \rightarrow 0^+$
also the modeling error introduced by using~\eqref{eq:approx_flux_pnp}
instead of~\eqref{eq:pforte_flux} decreases accordingly.

We now show that the equation system~\eqref{eq:pforte} in which
the flux~\eqref{eq:pforte_flux} is replaced by~\eqref{eq:approx_flux_pnp}
is uniquely solvable.
Let us introduce the following Hilbert space (see~\cite{QV} Chapter 1) 
\begin{equation}\label{eq:V_u}
V^{n_i} \equiv H^1_{\Gamma_D^{n_i}}(\Omega) = \{ v \in H^1(\Omega) : v|_{\Gamma_D^{n_i}} = 0 \}.
\end{equation}
Multiplying~\eqref{eq:pforte_cont} by 
a test function $v \in V^{n_i}$, integrating over $\Omega$ 
and enforcing the boundary conditions~\eqref{eq:pforte_bc_D} and~\eqref{eq:pforte_bc_N}, 
we obtain the weak formulation of the T-VE-PNP system: 
$\forall k = 0, \ldots, N-1$ and for all $j \geq 0$ until convergence
of the PNP cycle of Fig.~\ref{fig:gummelchart}:
\textit{find ${n}_i^{j+1}\in V^{n_i}$ such that}
\begin{align}\label{eq:vPpnp}
& \Frac{1}{\Delta t} 
( {n}_i^{j+1}, v ) +  a^{n_i}({n}_i^{j+1}, v) = 
F^{n_i}(v) & \forall v \in V^{n_i}
\end{align}
where:
\begin{equation*}
\begin{split}
\centering
({n}_i^{j+1},v) &= \int_{\Omega} {n}_i^{j+1} v \ d\Omega  \\
a^{n_i}({n}_i^{j+1},v) &= \int_{\Omega} D_i^{j} 
\nabla {n}_i^{j+1} \cdot 
\nabla v \ d\Omega + 
\int_{\Omega} {{q D_i^{j} z_i}
\over{k_B H_{\mathcal{T}_h}(T^{j})}} 
{n}_i^{j+1} \nabla \varphi^{j+1} \cdot 
\nabla v \ d\Omega \\
&+ \int_{\Omega} D_i^{j} {n}_i^{j+1}
\Frac{\nabla T^{j}}{T^{j}} \cdot \nabla v \ d\Omega - 
\int_{\Omega} {n}_i^{j+1} \mathbf{u}^k \cdot \nabla v \ d\Omega \\
F^{n_i}(v) &= \Frac{1}{\Delta t} \int_{\Omega} {n}_i^{k} v \ d\Omega.
\end{split}
\end{equation*} 
We notice that letting $\Delta t \rightarrow +\infty$ allows us 
to recover the steady-state problem. Using a standard technique 
in the analytical and computational study of the DD model for 
semiconductors, we introduce the change of dependent variable 
\begin{equation}\label{eq:changeVarpnp}
\tilde{n}_i = n_i e^{- \Psi}
\end{equation}
where $\Psi : \mathcal{T}_h \rightarrow \mathbb{R}$ is 
a dimensionless potential such that $\Psi|_K \in W^{1,\infty}(K)$
for each $K \in \mathcal{T}_h$, defined as
\begin{equation}\label{eq:psipnp}   
\Psi|_K = - \left[ \Frac{q z}{k_B H_K(T)} \varphi|_K 
+ \ln \left( \Frac{T|_K}{T_{ref}} \right) \right] \qquad \forall K \in \mathcal{T}_h
\end{equation}
$T_{ref}$ being a reference temperature (typically $T_{ref} = 300 \unit{K}$).
Substituting~\eqref{eq:changeVarpnp} into the expression for the
flux~\eqref{eq:pforte_flux} yields
\begin{equation}\label{eq:pnpflux_symm}
\mathbf{f}_{n_i}^{j+1} = - 
{D}_i^{j} e^{\Psi^{j}} \nabla \tilde{n}_i^{j+1} + 
\tilde{n}_i^{j+1}\mathbf{u}^k.
\end{equation}
The equivalent form of the flux shows that 
$\mathbf{f}_{n_i}^{j+1}$ is composed by the
sum of two terms, one representing the Fick law of diffusion of the
novel species $\tilde{n}_i^{j+1}$ with the 
modified diffusion coefficient ${D_i}^{j} e^{\Psi^{j}}$, the other 
representing passive convection of $\tilde{n}_i^{j+1}$ driven by the
electrolyte fluid velocity.

\begin{remark}
We notice that in the isothermal case the potential $\Psi$ coincides
with the electric potential $\varphi$ and the change of 
variable~\eqref{eq:changeVarpnp} is nothing but the standard use of
the so-called Slotboom variable very common in the study and numerical 
solution of the DD model 
(see~\cite{selberherr,markowich1986stationary,JeromeBook}).
\end{remark}
\noindent
We also make the hypothesis that, 
$\forall j \geq 0$ until convergence, 
$$
0 < {D}_i^{j,min} \leq D_i^{j}(\mathbf{x},t) \leq 
{D}_i^{j,max} \qquad \mbox{almost everywhere in } \Omega.
$$
Now, replacing~\eqref{eq:changeVarpnp} and~\eqref{eq:psipnp}
in~\eqref{eq:vPpnp}, we obtain the weak formulation of the T-VE-PNP
system:
$\forall k = 0, \ldots, N-1$ and for all $j \geq 0$ until convergence
of the PNP cycle of Fig.~\ref{fig:gummelchart}: \\
\textit{find $\tilde{n}^{j+1}_i$ $\in V^{n_i}$ such that}
\begin{equation}\label{eq:wppnp}
b^{{n}_i}(\tilde{n}_i^{j+1}, v) = F^{n_i}(v) \qquad \forall v \in V^{n_i}
\end{equation}
with
\begin{equation}\label{eq:bilinear_form_b}
\begin{split}
\centering
b^{{n}_i}(\mathcal{N},v) &= \Frac{1}{\Delta t} 
\int_{\Omega} \mathcal{N} e^{\Psi^{j}} v \ d\Omega + 
\int_{\Omega} {D}_i^{j} e^{\Psi^{j}} \nabla \mathcal{N} \cdot 
\nabla v \ d\Omega  \\
& - \int_{\Omega} \mathcal{N} e^{\Psi^{j}} 
\mathbf{u}^k \cdot \nabla v \ d\Omega \qquad \forall 
\mathcal{N},v \in V^{n_i}.
\end{split}
\end{equation}

\begin{theorem}\label{theo:pnp}
The weak problem~\eqref{eq:wppnp} is well posed and its solution 
satisfies the following stability estimate
\begin{equation}\label{eq:stability_estimate_pnp}
|| \tilde{n}^{j+1}_i ||_1 \leq e^{-\Psi^{min}} \Frac{||n^k_i||_0}
{\min \left\{1 , \Delta t D_i^{j+1,min} \right\}}.
\end{equation}
\end{theorem}
\begin{proof}
To prove the well-posedness of~\eqref{eq:vPpnp} 
we apply the Lax-Milgram Lemma (see ~\cite{QV}, Chapter 5).
We have:
\begin{subequations}\label{eq:LMconditions}
\begin{itemize}
\item $F^{n_i}(\cdot)$ is a continuous linear functional
\begin{align}
| F^{n_i}(v) | &= \Frac{1}{\Delta t} \left| 
\int_{\Omega} {n}_i^k v \ d\Omega\right| \leq 
\Frac{1}{\Delta t} ||n^k_i||_0 ||v||_0 &
\forall v \in V^{n_i}; \label{eq:continuity_F}
\end{align}

\item $b^{n_i}(\cdot,\cdot)$ is a continuous bilinear form
\begin{equation}\label{eq:continuity_b}
\begin{split}
| b^{n_i}(\mathcal{N},v) | & 
\leq \Frac{1}{\Delta t} e^{\Psi^{max}} || \mathcal{N} ||_0 || v||_0 + D^{j,max}_i e^{\Psi^{max}} ||\nabla \mathcal{N} ||_0 ||\nabla v||_0 \\
& + M e^{\Psi^{max}} || \mathcal{N} ||_1 ||\nabla v||_0  \\
&\leq \left(\Frac{1}{\Delta t}e^{\Psi^{max}} + 
D^{j,max}_i e^{\Psi^{max}} + M e^{\Psi^{max}} \right) 
|| \mathcal{N} ||_1 ||v||_1 \ \\ &\forall \ \mathcal{N}, v \in V^{n_i};
\end{split}
\end{equation}
\item $b^{n_i}(\cdot,\cdot)$ is a coercive bilinear form.
Given a function $f \in L^1(\Omega)$ we denote by
$f^+$ its positive part. Moreover, we indicate by $H$ 
the Heaviside function such that $H(0)=0$. We have
\begin{equation}\label{eq:coercivity_b}
\begin{split}
b^{n_i}(N, N) & \geq \Frac{1}{\Delta t}
e^{\Psi^{min}} || N ||_0^2 + 
D^{j,min}_i e^{\Psi^{min}} ||\nabla N ||_0^2 \\
& - e^{\Psi^{max}} 
\int_{\Omega} \left(N \mathbf{u}^k \cdot 
\nabla N \right)^+ \ d\Omega
\end{split}
\end{equation}
where $||\cdot ||_0$ is the $L^2$-norm, 
$||\cdot||_1 = ( \ ||\cdot ||_0^2 + 
||\nabla (\cdot) ||_0^2 \ )^{{1}\over{2}}$ is the $H^1$-norm.
Set $f = \left(N \mathbf{u}^k \cdot \nabla N \right)^+$
and notice that $f \in L^1(\Omega)$ by the Sobolev and 
H\"older inequalities.
Using the fact that ${\rm div \ } \mathbf{u}^k = 0$, 
a direct calculation yields
$$
{\rm div}\left( \Frac{1}{2} N^2 \mathbf{u}^k H(f) \right) = f,
$$
from which, using the fact that 
$\mathbf{u}^k = \mathbf{0}$ on $\Gamma_N^{n_i}$, we obtain
$$
\int_{\Omega} \left(N \mathbf{u}^k \cdot 
\nabla N \right)^+ \ d\Omega = 
\Int{\Omega}{} f \ d\Omega = \Int{\Omega}{} {\rm div} 
\left( \Frac{1}{2} N^2 \mathbf{u}^k H(f) \right) d \Omega = 0.
$$
Therefore,~\eqref{eq:coercivity_b} yields
\begin{equation}\label{eq:coercivity_b_2}
b^{n_i}(N, N) \geq e^{\Psi^{min}} 
\min \left\{\Frac{1}{\Delta t}, \, D_i^{j,min} \right\}
||N ||_1^2 \qquad \forall N \in V^{n_i}.
\end{equation}
\end{itemize}
\end{subequations}
Having verified that properties~\eqref{eq:LMconditions} are satisfied,
the use of the Lax-Milgram Lemma allows us 
to conclude that the weak problem~\eqref{eq:wppnp} is uniquely solvable.
The stability estimate~\eqref{eq:stability_estimate_pnp} immediately follows from~\eqref{eq:continuity_F} and~\eqref{eq:coercivity_b_2}.
\end{proof}


\subsection{Stokes system}\label{Stokes_well}

Applying the time discretization to~\eqref{eq:stpr} 
and enforcing the boundary conditions~\eqref{eq:stBC}, 
we arrive to the following stationary generalized Stokes problem:
\begin{equation}\label{eq:stokes_tdiscr}
\begin{cases}
\rho_f \displaystyle{{(\mathbf{u}^{k+1} - \mathbf{u}^k)}\over{\Delta t}} 
- \displaystyle {{\rm div \ } ( 2 \mu_f \underline{\underline{\epsilon}}(\mathbf{u}^{k+1}) - p^{k+1} \underline{\underline{ \delta}}) }= \mathbf{f}^{k+1}& \mbox{in } \ \Omega \\[3mm]
\displaystyle {{ \rm div \ } \mathbf{u}^{k+1}} = 0& \mbox{in } \ \Omega \\
\displaystyle {\underline{\underline{\epsilon}}({\mathbf{u}^{k+1}})  
}= {{1}\over{2}} (\nabla \mathbf{u}^{k+1} + (\nabla \mathbf{u}^{k+1})^T)\\
\mathbf{u}^{k+1} = \mathbf{0} & \mbox{on } \ \Gamma_D^{\mathbf{u}} \\ 
\displaystyle {\underline{\underline{\sigma}}(\mathbf{u}^{k+1}, p^{k+1})\mathbf{n} } 
= \mathbf{h}^{k+1} & \mbox{on } \ \Gamma_N^{\mathbf{u}}
\end{cases}
\end{equation}
where $\mathbf{f}^{k+1} := q \sum_{i=1}^M z_i n_i^{k+1} 
\mathbf{E}^{k+1}$, $n_i^{k+1}$ and 
$\mathbf{E}^{k+1}$ being given functions 
computed by solving the T-VE-PNP system~\eqref{eq:pforte}
and the Poisson equation~\eqref{eq:PoissonPNP}, respectively.

We set $V^{\mathbf{u}}: = [H^1_{0,\Gamma_D^{\mathbf u}}(\Omega)]^3$,
and $Q^p: = L^2(\Omega)$. Then, we multiply  
the first and the second equations of~\eqref{eq:stokes_tdiscr}
by a test function $\mathbf{v} \in V^{\mathbf{u}}$ and $q \in Q^p$, respectively, and integrate over 
$\Omega$, to obtain the following abstract problem: \\
$\forall \ t^k > 0, k = 0, ... , N-1, $ \textit{find 
$\mathbf{u}^{k+1} \in V^{\mathbf{u}}$ and $p^{k+1} \in Q^p$ such that:}
\begin{subequations}\label{eq:saddlePoint}
\begin{align}
& a^{\mathbf{u}}(\mathbf{u}^{k+1}, \mathbf{v}) + 
b^{\mathbf{u},p}(\mathbf{v}, p^{k+1}) = F^{\mathbf{u}}(\mathbf{v}) &
\qquad \forall \mathbf{v} \in V^{\mathbf u} \label{eq:momentum_st_weak} \\
& b^{\mathbf{u},p}(\mathbf{u}^{k+1}, q) = 0 & \qquad \forall q \in Q^p
\label{eq:mass_st_weak}
\end{align}
\end{subequations}
where:
$$
\begin{aligned}
&a^{\mathbf{u}}(\mathbf{u}^{k+1}, \mathbf{v}) 
=\int_{\Omega}{{\mathbf{u}^{k+1} \cdot  \mathbf{v}}\over{\Delta t}} \ d\Omega +\int_{\Omega} 2\nu \underline{\underline{\epsilon}}(\mathbf{u}^{k+1}) : \underline{\underline{\epsilon}}(\mathbf{v}) \ d\Omega \\
&b^{\mathbf{u},p}(\mathbf{v}, p^{k+1}) = 
- \int_{\Omega} p^{k+1} \, {\rm div \ } \mathbf{v} \ d\Omega \\
&F^{\mathbf{u}}(\mathbf{v}) =  \int_{\Omega} {{\mathbf{f}^{k+1} \cdot  \mathbf{v}}\over{\rho_f}} \ d\Omega + \int_{\Omega} {{\mathbf{u}^k \cdot  \mathbf{v}}\over{\Delta t}} \ d\Omega + \int_{\Gamma_N} {{\mathbf{h} \cdot \mathbf{v}}\over{\rho_f}} \ d\Gamma. 
\end{aligned}
$$
\begin{proposition}\label{prop:stokes}
The weak problem~\eqref{eq:saddlePoint} is well posed and its solution 
satisfies the following stability estimate
\begin{equation}\label{eq:stability_estimate_stokes}
|| \mathbf{u}^{k+1} ||_{V^{\mathbf u}} \leq C_{\mathbf u}, \qquad 
|| p^{k+1} ||_{Q^p} \leq C_p
\end{equation}
where $|| \mathbf{v} ||_{V^{\mathbf u}} = 
|| \underline{\underline{\epsilon}}
({\mathbf{v}}) ||_0$ and $|| q ||_{Q^p} = ||q ||_0$ for all
$\mathbf{v} \in V^{\mathbf{u}}$ and $q \in Q^p$, 
while $C_{\mathbf u}$ and $C_p$ are positive constants depending on problem data, on the coercivity constant of the
bilinear form $a^{\mathbf{u}}(\cdot, \cdot)$ on ${\rm ker}\,V^{\mathbf u}$ 
and on the inf-sup constant of the bilinear form $b^{\mathbf{u},p}
(\cdot, \cdot)$.
\end{proposition}
\begin{proof}
Well-posedness of~\eqref{eq:saddlePoint} can be proved
using the abstract theory for 
saddle-point problems developed in~\cite{QV}, Section 7.4.
The stability estimate~\eqref{eq:stability_estimate_stokes} is a consequence
of the application of Theorem 7.4.1 of~\cite{QV}.
\end{proof}

\subsection{The mechanical problem} \label{Mech_well}
%
%
%
%
%
To proceed with the weak formulation of the Navier-Lam\'e model in presence of a thermal field
we set 
\begin{align}
& V^{\mathbf{d}}: = [H^1_{0,\Gamma_D^{\mathbf d}}(\Omega)]^3, 
& \label{eq:space_V_mech}
\end{align}
multiply~\eqref{eq:MECeq} by a test function 
$\mathbf{v} \in V^{\mathbf{d}}$ and integrate over 
$\Omega$. Applying the boundary conditions~\eqref{eq:bcs_mech}, 
the weak formulation of the mechanical problem is:\\
\textit{find $\mathbf{d}^{k+1} \in V^{\mathbf{d}}$ such that:}\\
\begin{equation}\label{eq:weak_mech}
a^{{\mathbf{d}}}(\mathbf{d}^{k+1}, \mathbf{v}) = 
F^{{\mathbf{d}}}(\mathbf{v})  \qquad \forall \mathbf{v} \in V^{{\mathbf{d}}}
\end{equation}
where:
$$
\begin{aligned}
&a^{{\mathbf{d}}}(\mathbf{d}^{k+1}, \mathbf{v}) = 
\int_{\Omega}  \underline{\underline{\textit{C}}} 
\underline{\underline{\varepsilon}}_{mec} (\vect{d}^{k+1}) : 
\underline{\underline{\epsilon}}(\mathbf{v}) \ d\Omega \\
&F^{{\mathbf{d}}}(\mathbf{v}) =   \int_{\Omega}  \underline{\underline{\textit{C}}} 
\underline{\underline{\varepsilon}}_{mec}^{th} (T^{k}) : 
\underline{\underline{\epsilon}}(\mathbf{v}) \ d\Omega + \int_{\Omega} 
{ {\rm div \ } \underline{\underline{\sigma_0}} \cdot  \mathbf{v}} \ d\Omega + \int_{\Omega} {\mathbf{f_{mec}} \cdot  \mathbf{v}} \ d\Omega.
\end{aligned}
$$

\begin{proposition}\label{prop:mech}
The weak problem~\eqref{eq:weak_mech} is well posed and its solution 
satisfies the following stability estimate
\begin{equation}\label{eq:stability_estimate_mech}
|| \mathbf{d}^{k+1} ||_{V^{\mathbf d}} \leq C_{\mathbf d}
\end{equation}
where $|| \mathbf{v} ||_{V^{\mathbf d}} = || 
\underline{\underline{\epsilon}}
({\mathbf{v}}) ||_0$ and $C_{\mathbf d}$ is a positive constant depending on problem data 
and on the coercivity constant of the
bilinear form $a^{{\mathbf{d}}}(\cdot, \cdot)$.
\end{proposition}
\begin{proof}
The proof follows from the application of Lax-Milgram Lemma 
and Korn's inequality (see, e.g.,~\cite{braessbook}, Thm. 3.4).
\end{proof}
\
\section{Galerkin Finite Element Approximation}
\label{sec:feapproximation}

In this section we discuss the Galerkin Finite Element approximation
of each weak problem introduced and analyzed in sec.~\ref{sec:weak}.
In sec.~\ref{TVEPNP_discr} we study the T-VE-PNP system,
the Stokes system in sec.~\ref{Stokes_discr}, 
and the Navier Lam\'e formulation for the mechanical displacement
in sec.~\ref{Mech_discr}. Throughout this section, for $r \geq 1$, 
we denote by $\mathbb{P}_r(K)$ the space of algebraic polynomials 
of degree $\leq r$ defined on the element $K \in \mathcal{T}_h$.

\subsection{T-VE-PNP system}\label{TVEPNP_discr}

Let us define the following finite-dimensional subspace 
of~\eqref{eq:V_u}
$$
V_h^{n_i} = 
\{ v_h \in C^0({\overline{\Omega}}) : v_h|_K \in \mathbb{P}_1(K) 
\, \, \forall K \in \mathcal{T}_h,\  v_h|_{\Gamma_D^{n_i}} = 0 \}
$$
such that $N_h = {\rm dim}(V^{n_i}_h)$. The Galerkin Finite Element approximation
of~\eqref{eq:wppnp} in matrix form is the linear algebraic system
%
%
%
\begin{equation}\label{eq:alsys}
K^{n_i} \widetilde{\mathbf{n}}^{k+1} = \mathbf{F}^{n_i}
\end{equation}
where $K^{n_i} = M^{n_i} + B^{n_i}$ is the generalized stiffness matrix,
$M^{n_i}$ is the mass matrix, $B^{n_i}$ is the stiffness matrix 
and $\mathbf{F}^{n_i}$ is the load vector, given by:
\begin{subequations}\label{eq:matrices_n}
\begin{align}
& M^{n_i} = \Frac{1}{\Delta t} (e^{\Psi} \Phi_j, \Phi_i) & \label{eq:mass_matrix} \\[2mm]
& B^{n_i} = b_h^{n_i}(\Phi_j, \Phi_i) & \label{eq:stiffness_matrix} \\[2mm]
& \mathbf{F}^{n_i} = F^{n_i}(\Phi_i). & \label{eq:load_vector}
\end{align}
\end{subequations}
In~\eqref{eq:matrices_n}, $\{ \Phi_i \}_{i=1}^{N_h}$ is the Lagrangian set of basis functions for $V_h^{n_i}$ whereas $b_h^{n_i}(\cdot, \cdot)$ 
is a suitable modification of the bilinear form~\eqref{eq:bilinear_form_b}.
The solution vector $\widetilde{\mathbf{n}}^{k+1}$ of the linear system~(\ref{eq:alsys}) contains the degrees of freedom (dofs)
of the approximation $\widetilde{n}_{i,h}^{k+1}$ of the 
unknown ion concentration $\widetilde{n}_i$ at time $t^{k+1}$.

Three computational issues are worth noting in the 
stable and accurate solution of system~\eqref{eq:alsys}.
The first issue concerns the mass matrix $M^{n_i}$. We compute its 
entries using the 3D trapezoidal rule which corresponds to 
"lumping" diagonalization of $M^{n_i}$. The second issue concerns the 
numerical quadrature rule used to compute the entries 
$b_h(\Phi_j, \Phi_i)$ of the stiffness matrix $B^{n_i}$. 
As a matter of fact, if the convective term of the 
flow $\mathbf{f_{n_i}}$ dominates the diffusive contribution, 
the standard finite element method with piecewise 
linear polynomials may lead to instability, with spurious oscillations in 
the numerical results. An established remedy to this drawback 
is the EAFE (Edge Averaged Finite Element)
approach studied in 2D in~\cite{gatti1998new} and in 3D in~\cite{Zikatanov:EAFE}. 
This latter is a multidimensional generalization of the 
celebrated 1D Scharfetter-Gummel method~\cite{SG}, and 
is obtained by replacing along each edge of the element 
$K$ the diffusion coefficient $D_i^{k+1} e^{\Psi}$ 
in~\eqref{eq:bilinear_form_b} with its harmonic average. 
The third issue concerns the choice of the dependent variable
in the computer implementation of the linear system~\eqref{eq:alsys}.
As a matter of fact, the change of variable~\eqref{eq:changeVarpnp}
gives rise to a symmetric positive definite stiffness matrix 
$B^{n_i}$. However, the entries of this matrix may be impossible
to compute because of overflow exceptions due to the dynamic range
required by the evaluation of the quantity $e^{\Psi}$. Thus, it is
necessary to go back to the original dependent variable $n$ by 
applying the inverse of the transformation~\eqref{eq:changeVarpnp}
at each node of $\mathcal{T}_h$. This is equivalent to 
multiplying each column $I$ of $B^{n_i}$ by $e^{-\Psi_I}$, for 
$I=1, \ldots, N_h$ (cf.~\cite{Brezzisiam1989,Brezzicmame1989}).
The novel linear algebraic system for the variable $n$ can be written 
in matrix form as
\begin{equation}\label{eq:alsys_n}
K^{n} \mathbf{n}^{k+1} = \mathbf{F}^{n_i}
\end{equation}
where 
$$
K^{n} = K^{n_i} {\rm diag}(e^{-\Psi}).
$$
Solving~\eqref{eq:alsys_n} yields
an accurate and robust treatment of
the mass balance equations for $n_i$ especially in the presence of
high convective terms ($\mathbf{E}$ and/or $\mathbf{u}$), preventing
the occurrence of spurious unphysical oscillations in the computed
ion concentrations under suitable conditions on the regularity
of $\mathcal{T}_h$. 
\begin{proposition}[Discrete maximum principle]\label{prop:DMP}
Assume $\mathbf{u} = \mathbf{0}$ and that
the triangulation $\mathcal{T}_h$ satisfies 
Lemma 2.1 of~\cite{Zikatanov:EAFE}. Then, 
$K^{n}$ is a diagonally dominant by columns M-matrix. 
This implies that system~\eqref{eq:alsys_n} 
is uniquely solvable~\cite{varga1999matrix} 
and that if $\mathbf{F}^{k+1} \geq 0$ then 
$\mathbf{n}^{k+1} > 0$ (Discrete Maximum Principle, DMP).
\end{proposition}

The previous result shows the beneficial properties of the 
extension of the EAFE method to the numerical approximation of the 
PNP-T model. The analogue of Prop.~\ref{prop:DMP} in the case 
of the VE-PNP-T model (where the electrolyte fluid velocity 
is non vanishing) is at moment an open 
issue that will be the object of a next step of the current research.
We limit ourselves to the observation that in all the numerical experiments
reported in sec.~\ref{sec:simulations} 
no spurious oscillations ever appeared in the
computed ion concentrations. This is to be ascribed to
the fact that $\mathbf{u}$ is a small perturbation to 
thermo-electrochemical flow so that, in practice, the method
continues to satisfy a DMP.

\subsection{Stokes system} \label{Stokes_discr}

In sec.~\ref{sec:stokes_fem} we describe the two finite 
element schemes that are used in MP-FEMOS for the 
numerical discretization of the saddle-point 
problem~\eqref{eq:saddlePoint} whereas in sec.~\ref{sec:stokes_alg}
we illustrate and analyze 
the algorithms that are implemented for the efficient
solution of the linear algebraic system.

\subsubsection{Finite element schemes}\label{sec:stokes_fem}
The Galerkin finite element approximation of 
the generalized Stokes problem ~\eqref{eq:saddlePoint} 
is a delicate issue because if the finite element spaces for 
the electrolyte fluid velocity and pressure do not satisfy a
compatibility condition (the so-called inf-sup or LBB 
condition~\cite{Brezzi1974}) then uniqueness of the discrete solution 
fails to hold and spurious pressure modes may arise 
affecting the stability of the formulation. To overcome this difficulty
two general approaches can be adopted. The first approach consists in 
selecting a finite element pair satisfying the inf-sup condition 1) in 
a manner that is independent of the discretization parameter $h$; and
2) with an optimal convergence order in the graph norm with respect
to the $V \times Q$ topology. 
The second approach consists in modifying the saddle point 
problem~\eqref{eq:saddlePoint} by the introduction of a stabilization
term whose weight is strong enough to circumvent the inf-sup condition. 

The first approach is pursued in MP-FEMOS by implementing the so-called
Taylor-Hood pair (see~\cite{QV} Chapter 9 for a detailed description 
and analysis). This amounts to setting:
\begin{subequations}\label{spaces_stokes}
\begin{align}
& V_{h}^{TH} = \{ \mathbf{v}_h \in [
C^0(\overline{\Omega})]^3: 
\mathbf{v}_h|_K \in [\mathbb{P}_2 (K)]^3 \ \forall K \in \mathcal{T}_h, \mathbf{v}_h|_{\Gamma_D^{\mathbf{u}}} = 0 \} & \label{space_velocity} \\
& Q_h^{TH} = \{ q_h \in C^0(\overline{\Omega}): 
q_h|_K \in \mathbb{P}_1 (K) \ \forall K \in \mathcal{T}_h \} 
& \label{space_pressure}
\end{align}
and to solving the following sequence of saddle-point problems: \\
$\forall \ t^k > 0, k = 0, ... , N-1$, \textit{find $\mathbf{u}^{k+1}_h 
\in V_{h}^{TH}$ and $p_h^{k+1} \in Q_h^{TH}$ such that:}
\begin{align}
&a^{\mathbf{u}}(\mathbf{u}^{k+1}_h , \mathbf{v}_h ) + 
b^{\mathbf{u},p}(\mathbf{v}_h , p_h ) = 
F^{\mathbf{u}}(\mathbf{v}_h) & \qquad \forall \mathbf{v}_h \in V_{h}^{TH} 
\label{eq:TH_momentum} \\
& b^{\mathbf{u},p}( \mathbf{u}^{k+1}_h , q_h ) = 0 & \qquad 
\forall q_h \in Q_h^{TH}. \label{eq:TH_mass}
\end{align}
\end{subequations}
The second approch is pursued in MP-FEMOS by implementing the so-called
Hughes-Franca-Balestra stabilization (see~\cite{HughesFrancaBalestra1986}).
This amounts to setting:
\begin{subequations}\label{spaces_stokes_HFB}
\begin{align}
& V_{h}^{HFB} = \{ \mathbf{v}_h \in [
C^0(\overline{\Omega})]^3: 
\mathbf{v}_h|_K \in [\mathbb{P}_1 (K)]^3 \ \forall K \in \mathcal{T}_h, \mathbf{v}_h|_{\Gamma_D^{\mathbf{u}}} = 0 \} & \label{space_velocity_HFB} \\
& Q_h^{HFB} = \{ q_h \in C^0(\overline{\Omega}): 
q_h|_K \in \mathbb{P}_1 (K) \ \forall K \in \mathcal{T}_h \} 
& \label{space_pressure_HFB}
\end{align}
and to solving the following sequence of 
stabilized saddle-point problems: \\
$\forall \ t^k > 0, k = 0, ... , N-1$, 
\textit{find \ $\mathbf{u}^{k+1}_h \in V_{h}^{HFB}$ 
and $p_h^{k+1}\in Q_h^{HFB}$ such that:}
\begin{align}
& a^{\mathbf{u}}(\mathbf{u}^{k+1}_h,\mathbf{v}_h) + 
b^{\mathbf{u},p}(\mathbf{v}_h, p_h) = 
F^{\mathbf{u}}(\mathbf{v}_h) & \quad \forall 
\mathbf{v}_h  \in V_{h}^{HFB} \label{eq:momentum_HFB} \\
& b^{\mathbf{u},p}(\mathbf{u}^{k+1}_h,q_h) = \Phi_h^{HFB}(q_h) & \qquad 
\forall q_h \in Q_h^{HFB} \label{eq:mass_HFB}
\end{align}
where:
\begin{align*}
& \Phi_h^{HFB}(q_h) :=  
\sum_K {{ h_K^2}\over{\delta}} \int_{K} 
\left( 
{{1}\over{\Delta t}}\mathbf{u}^{k+1}_h - 
\nu \Delta \mathbf{u}^{k+1}_h + \nabla p_h - 
\mathbf{\tilde{f}}
\right) \cdot \nabla q_h & \\
& \displaystyle {\mathbf{\tilde{f}} = {{\mathbf{f}}\over{\rho_f}} + {{\mathbf{u}^{k}_h}\over{\Delta t}}}. & 
\end{align*}
The quantity $\delta$ is 
a positive parameter to be properly selected upon noting that 
a too large value of $\delta$ does not eliminate the spurious modes of the pressure while a too small value of $\delta$ yields a poor approximation 
for the pressure field near the domain boundary 
(cf.~\cite{QV}, Chapter 9).
\end{subequations}

\subsubsection{Analysis and linear solvers for the generalized Stokes
system}\label{sec:stokes_alg}

In this section we assume that 1) the 
time advancing scheme has second-order accuracy (as in the case of the 
Crank-Nicolson scheme or the TR-BDF2 scheme); and 2) the exact solution
is appropriately smooth. The quantity $C$ denotes a positive constant
independent of $h$ whose value is not the same at each occurrence.

The Taylor-Hood FE solution 
$\mathbf{u}_h^{TH} \in V^{TH}_h \times p_h^{TH} \in Q_h^{TH}$ satisfies
the following optimal error estimate (see~\cite{QV}, Chapter 9)
\begin{align}
& || \mathbf{u}^{k+1} - \mathbf{u}_h^{TH} ||_V + 
|| p^{k+1} - p_h^{TH} ||_Q \leq C h^2. & \label{eq:error_TH}
\end{align}
In spite of the fact that~\eqref{eq:error_TH} shows that the TH method is 
second-order accurate, the linear system associated with 
the saddle-point FE problem~\eqref{eq:TH_momentum}-~\eqref{eq:TH_mass}
has a very large size in 3D computations, so that it is a good practice 
to use an iterative method for its solution. 
In MP-FEMOS, the adopted approach is the Uzawa method (see~\cite{QV}
Chapter 9) with preconditioner $P$ given by 
the mass matrix scaled by the fluid viscosity $\nu$ and acceleration
parameter $\rho$ such that $\rho > 1/(2 \nu)$. 
In the case where convergence of Uzawa's iteration is still to slow, an alternative is given by the 
MUltifrontal Massively Parallel sparse direct Solver (MUMPS) 
developed at INRIA~\cite{MUMPS}.


The HFB FE solution 
$\mathbf{u}_h^{HFB} \in V^{HFB}_h \times p_h^{HFB} \in Q_h^{HFB}$ 
satisfies the following error estimate (see~\cite{QV}, Chapter 9)
\begin{align}
& || \mathbf{u}^{k+1} - \mathbf{u}_h^{HFB} ||_V + 
|| p^{k+1} - p_h^{HFB} ||_Q \leq C h. & \label{eq:error_HFB}
\end{align}
The above result is suboptimal for the pressure variable. However, the
HFB stabilization, compared to the TH approach, 
has the advantage of allowing the use of equal-order 
interpolation for \emph{both} the unknowns of the Stokes system.
This reflects into a simpler coding, a much reduced
memory size and, more importantly, into a substantially lower 
computational effort. The disadvantage lies principally into
the need of a proper selection of the stabilization parameter
$\delta$.

A thorough comparison between the performance of the 
two solution methods considered in the MP-FEMOS software is addressed
in sec.~\ref{sec:simulations}.

\subsection{The mechanical problem} \label{Mech_discr}

Let us define the following finite-dimensional subspace 
of~\eqref{eq:space_V_mech}
$$
V_{h}^{\mathbf{d}} = 
\{ \mathbf{v}_h \in [C^0(\overline{\Omega})]^3: 
\mathbf{v}_h|_K \in [\mathds{P}_1(K)]^3 
\ \forall K \in \mathcal{T}_h, \mathbf{v}_h|_{\Gamma_D^{\mathbf{d}}} 
= \mathbf{0} \}.
$$
%
The Galerkin approximation of~\eqref{eq:weak_mech} is:\\
\textit{Find $ \mathbf{d}_h^{k+1} 
\in V_{h}^{\mathbf{d}}$ such that}:
\begin{equation} \label{eq:dispGProb} 
a^{\mathbf{d}}(\mathbf{d}_h^{k+1},\mathbf{v}_h) = F^{\mathbf{d}}(\mathbf{v}_h) \quad 
\forall \mathbf{v}_h \in V_{h}^{\mathbf{d}}.
\end{equation}
Problem~\eqref{eq:dispGProb} is equivalent to the solution of the
linear algebraic system
%
\begin{equation}\label{eq:mec_bloc}
A^{\mathbf{d}} \mathbf{d}^{k+1} = \mathbf{L}^{\mathbf d}
\end{equation}
where:
\begin{align*}
& A_{ij}^{\mathbf{d}} = 
a^{\mathbf{d}}(\bm{\Phi}_j ,\bm{\Phi}_i) & \qquad i,j = 1, ... , N_T\\
& \mathbf{L}_i^{\mathbf{d}} = \mathbf{F}^{\mathbf{d}}
(\bm{\Phi}_i) & \qquad i = 1, ... , N_T, 
\end{align*}
$N_T$ being the number of dofs for each component of the displacement
vector $\mathbf{d}$. The stiffness matrix $A^{\mathbf{d}}$ is symmetric
and positive definite from which it follows that~\eqref{eq:mec_bloc} is
uniquely solvable. 
No ill-conditioning due to material compressibility affects the 
linear elastic mechanical problem 
because the material Poisson ratio $\nu$ is far from the limit value
$0.5$ (cf. Tab.~\ref{tab:PNPdata}). Using standard arguments
(see~\cite{hughesbook,brenner2002mathematical}) the  following
error estimate can be proved
$$
|| \mathbf{u} -  \mathbf{u}_h||_{V^{\mathbf d}} \leq C_{\mathbf d} h
$$
where $C_{\mathbf d}$ is a positive constant independent of $h$ and 
depending only on the mesh aspect ratio $\gamma$, on problem data
and on the exact solution of~\eqref{eq:weak_mech}.
The numerical solution of the linear 
algebraic system~\eqref{eq:mec_bloc} is efficiently carried out 
in the MP-FEMOS platform through the use of the Bi-CG iterative 
method (see~\cite{quarteroni2007numerical}, Chapter 4 for a discussion 
of the algorithm and its convergence analysis).

\section{Simulation Results}\label{sec:simulations}

In this section we present the results of the reciprocal interaction 
among the thermal and mechanical aspects and the electrodiffusion 
of ions in a nanoscale biological channel. 
The considered case study is the same channel already investigated
in two spatial dimensions in~\cite{MM8} and, more recently, in three spatial dimensions in~\cite{Airoldi2015}. The biological setting is constituted by
a voltage operated channel in which K$^+$ (potassium) 
and Na$^+$ (sodium) ions are taken into account.
In the first step of our analysis a numerical simulation of the VE-PNP model 
is performed to investigate the coupling effects between the PNP and the Stokes 
systems in a 3D cylindrical geometry used 
as a natural extension of a 2D domain as already demonstrated in~\cite{Airoldi2015}.
In the second step of our analysis we carry out the study of the influence 
of the thermal gradient in the VE-PNP model. 
In the third step of our analysis a mechanical stress 
(with different strengths) is applied at the center of 
the cylinder causing a channel deformation. 
Under such a condition we investigate ion electrodiffusion using the 
VE-PNP model to highlight the variation in ion flow 
induced by channel restriction. 
For completeness of information, Tab.~\ref{tab:PNPdata} summarizes 
the values of all model parameters and also the 
initial and boundary conditions for the VE-PNP system.
\begin{table}\label{tab:PNPdata}
\centering
\begin{tabular}{|l|l|}
\hline
\textbf{Parameter}       &\textbf{value and units}  \\
\hline
$z_{K^+}$		 &$+1$\\
\hline
$z_{Na^+}$	 	 &$+1$\\
\hline
$\overline{K^+}|_{SideA}$	&$2.41 \cdot 10^{20} \, \unit{cm^{-3}}$\\
\hline
$\overline{Na^+}|_{SideA}$  &$3.01 \cdot 10^{19} \, \unit{cm^{-3}}$\\
\hline
$\overline{K^+}|_{SideB}$		&$1.2 \cdot 10^{19} \, \unit{cm^{-3}}$\\
\hline
$\overline{Na^+}|_{SideB}$	&$2.65 \cdot 10^{20} \, \unit{cm^{-3}}$\\
\hline
$\mu_{K^+}$		&$7.2 \cdot 10^{-4} \, \unit{cm^2 V^{-1} s^{-1}}$\\
\hline
$\mu_{Na^+}$		&$5.2 \cdot 10^{-4} \, \unit{cm^2 V^{-1} s^{-1}}$\\
\hline
$C^{init}_{K^+}$    &$1.2 \cdot 10^{19} \, \unit{cm^{-3}}$\\
\hline
$C^{init}_{Na^+}$	&$3.01 \cdot 10^{19} \, \unit{cm^{-3}}$\\
\hline
$\overline{\varphi}|_{SideA}$		 &$0.02 \, \unit{V} \text{ in sec.~\ref{sez1} and~\ref{sez2}}; 0.2 \, \unit{V} \text{ in sec.~\ref{sez3}}$\\
\hline
$\overline{\varphi}|_{SideB}$		 &$0 \, \unit{V}$	\\
\hline
$\overline{p}|_{SideA}$    &$0 \, \unit{N cm^{-2}}$\\
\hline
$\overline{p}|_{SideB}$    &$100 \, \unit{N cm^{-2}}$	\\
\hline
$\mu_f$    &$10^{-7} \, \unit{N s cm^{-2}}$	\\
\hline
$\rho_f$    &$10^{-3} \, \unit{Kg cm^{-3}}$	\\
\hline
$E$ in $\Omega_1$ and $\Omega_3$    &$1.5 \cdot 10^{7} \, 
\unit{N cm^{-2}}$	\\
\hline
$E$ in $\Omega_2$    &$1.5 \cdot 10^{7} \, \unit{N cm^{-2}}$	\\
\hline
$E$ in Cylinder    &$7 \cdot 10^{6} \, \unit{N cm^{-2}}$	\\
\hline
$\nu$ in $\Omega_1$ and $\Omega_3$    &$ 0.3 \  \ $	\\
\hline
$\nu$ in $\Omega_2$    &$ 0.4 \  \ $	\\
\hline
$\nu$ in Cylinder    &$ 0.2 \  \ $	\\
\hline
$T$    &$293.75 \, \unit{K}$	\\
\hline
\end{tabular}
\caption{VE-PNP model parameters in presence of 
thermal and mechanical forces and 
boundary conditions for Poisson and electrodiffusion equations.}
\end{table} 

In Figs.~\ref{fig:cil} and~\ref{fig:domMEC} 
the mesh structures used in the numerical simulations
are reported. The typical number of simplices is 
in the range of 50000 elements. 
Channel length is set at $10 \unit{nm}$ while channel diameter is equal to $2 \unit{nm}$. 
Channel terminals are restricted to the internal cylinder 
of the lateral vertical surfaces: 
the outside structure represents the biological region that is 
assigned to the constriction or to the dilation of the channel itself. 
The Z-axis is aligned with the symmetry axis of the structures and is
oriented from $SideA$ towards $SideB$ terminals.
Cell interior is located at $SideA$ while cell exterior ambient 
is located at $SideB$.
\begin{figure}[htbp]
\centering
\subfigure[\label{fig:cil}]
{\includegraphics[width=0.4\columnwidth]{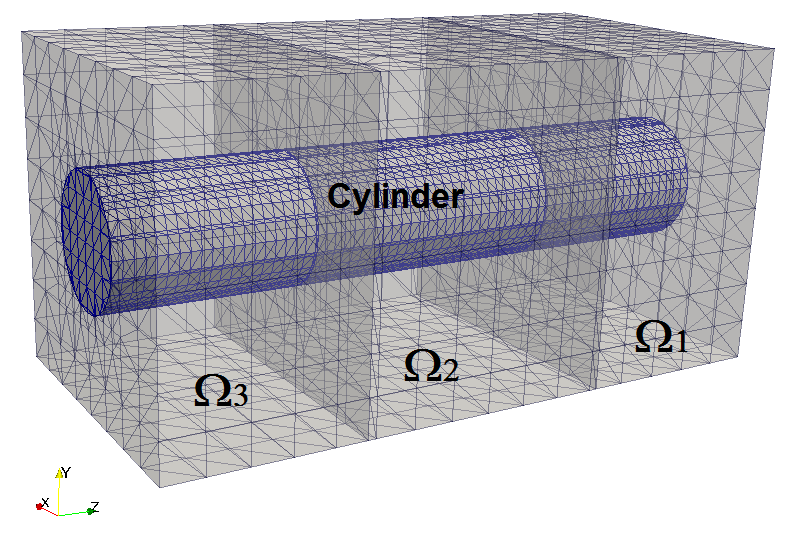}}
\subfigure[\label{fig:domMEC}]
{\includegraphics[width=0.5\columnwidth]{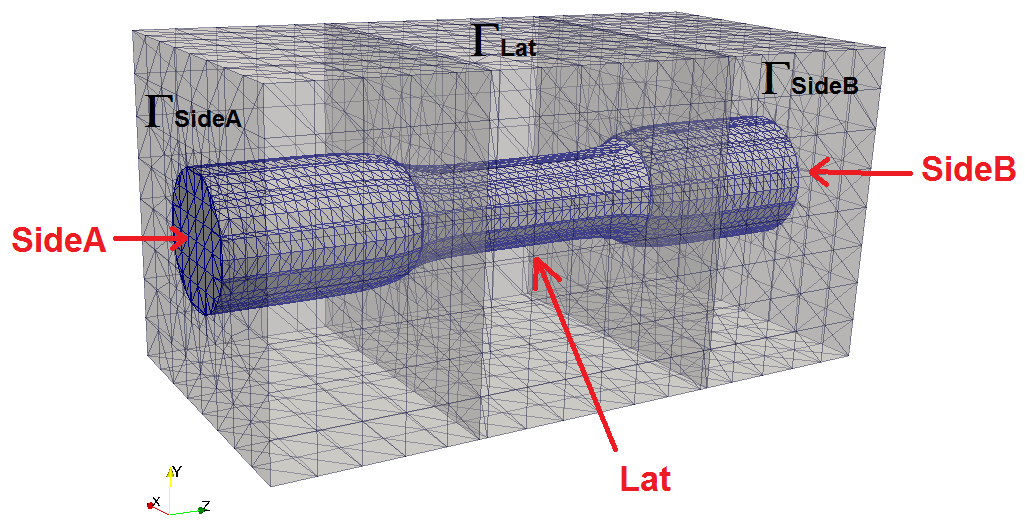}}
\caption{Mesh used in the simulations: (a) channel with no deformation; 
(b) channel with deformation due to the application of 
an internal stress in $\Omega_2$.}
\end{figure}

\subsection{Numerical results for the VE-PNP model}\label{sez1}  

In this case study we investigate
the impact of the velocity term in the VE-PNP model 
at a constant room temperature.
Referring to~\eqref{eq:pnpBC} and~\eqref{eq:stBC}, and according to Tab.~\ref{tab:PNPdata}, Dirichlet boundary conditions 
for ionic concentrations and electrostatic potential are enforced 
on $SideA$ (cell-inside) and $SideB$ (cell-outside) 
surfaces while homogeneous Neumann conditions hold elsewhere.
In the Stokes problem~\eqref{eq:stpr} Neumann boundary 
conditions are applied on $SideA$ and $SideB$,
with a pressure drop of $800 \, \unit{N cm^{-2}}$ 
whereas homogeneous Dirichlet conditions for fluid velocity 
hold elsewhere (adherence condition). 
The initial values are set at $C_{K^+}^{init}$ 
and $C_{Na^+}^{init}$ for the ionic concentrations and equal to
zero for the electrolyte fluid velocity.
\begin{figure}[htbp]
\centering
\subfigure[{Electrostatic potential.}\label{fig:3Dtestpot}]
{\includegraphics[width=0.48\textwidth]{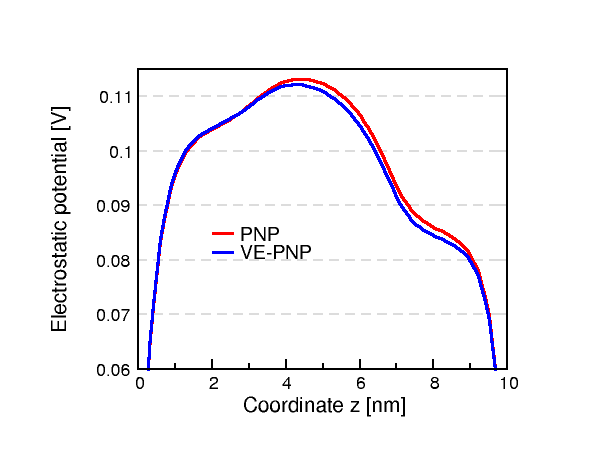}}
\subfigure[{Electric field.}\label{fig:EFM_modello}]
{\includegraphics[width=0.48\textwidth]{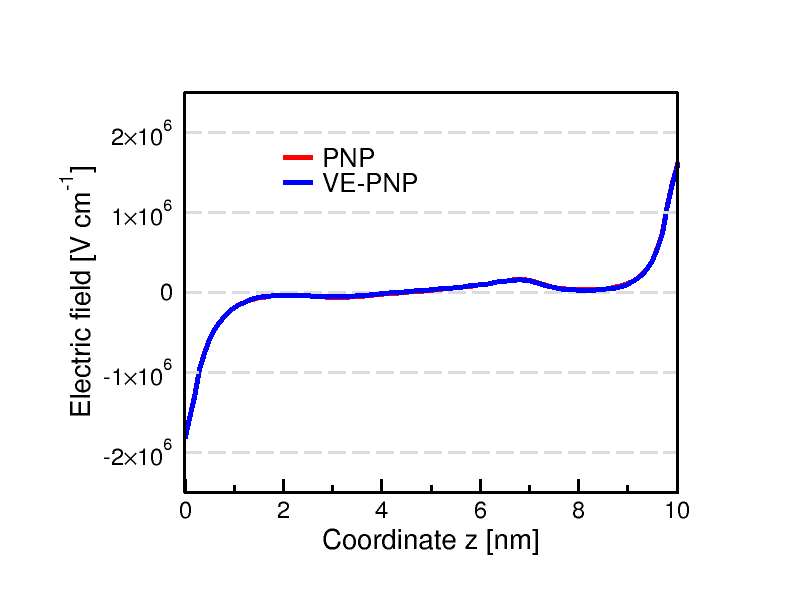} }
\caption{Comparison between VE-PNP and PNP models of the 
electrostatic potential (a) and electric field (b). 1D cut is along 
the channel axis at $t=50 \unit{ns}$.}
\label{fig:modello}
\end{figure}

Fig.~\ref{fig:3Dtestpot} and fig.~\ref{fig:EFM_modello} 
show at time $t=50 \unit{ns}$ the profile of electrostatic potential
and of the $z$ component of the 
electric field obtained along the cylinder 
symmetry axis for both VE-PNP and PNP-only models.
The onset of pronounced bumps in the electrostatic potential
is responsible for a nonuniform contribution of the 
electric field to ion electrodiffusion. As a matter of fact, 
close to the interior part of the channel drift forces are 
directed towards cell interior while at the end of the 
channel the drift forces point outward of the cell.
The difference between the prediction of the VE-PNP and PNP models
is not so relevant in this case, as confirmed by the 
fact that the small variations in the potential plot completely
disappear in the electric field plot.
\begin{figure}[h!]
\centering
\includegraphics[width=0.75\columnwidth]{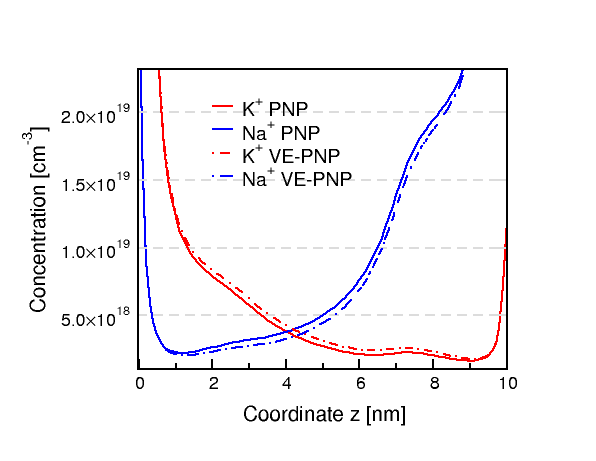}
\caption{Comparison between PNP and VE-PNP models of the 
K$^+$ and Na$^+$ ionic concentrations. 1D cut is along 
the channel axis at $t=50 \unit{ns}$.}
\label{fig:onlyPNP1}
\end{figure}

Fig.~\ref{fig:onlyPNP1} shows the profile of K$^+$ and Na$^+$ 
ion concentrations. Since the $z$-component of 
the electrolyte fluid velocity is positive the K$^+$ species 
is more diffused towards the positive $z$-axis in the case of the 
VE-PNP model.
Conversely, in the case of Na$^+$ the electrolyte fluid flow causes 
a reduction of diffusion towards cell interior.
The asymmetry in the ionic concentration 
profiles as shown in fig.~\ref{fig:onlyPNP1}
is to be ascribed to the cumulative effect of drift and diffusion
forces: in fact, close to $SideA$ and $Side B$ drift is 
opposing to diffusion but 
K$^+$ ions have a higher diffusion gradient and lower drift force.  
In fig.~\ref{fig:3Dtestvelocitycyl} a 
vector field plot of the electrolyte fluid velocity
is shown at time $t=50 \unit{ns}$: as anticipated, 
velocity is directed towards the positive $z$-axis
due to the imposed boundary conditions. 
\begin{figure}[h!]
\centering
\includegraphics[width=0.75\columnwidth]{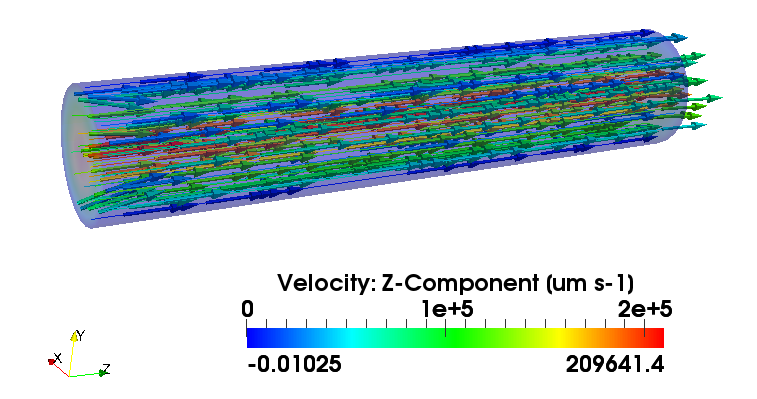}
\caption{VE-PNP model: 3D view of electrolyte fluid velocity 
vector field inside the channel at $t=50 \unit{ns}$.}
\label{fig:3Dtestvelocitycyl}
\end{figure}


\subsubsection{Validation of the Hughes-Franca-Balestra formulation} \label{sez1.1}

As already pointed out in sec.~\ref{Stokes_discr} the use of a stabilized
finite element space in the discretization of~\eqref{eq:stpr} 
considerably reduces the number of dofs and, as a consequence, the 
computational cost. The HFB methodology, however, requires 
an appropriate choice of the stabilization parameter $\delta$. This issue
is documented in the example below.
Fig.~\ref{fig:P1stabvel} and fig.~\ref{fig:P1stabp} show
the velocity and the pressure profiles 
in a case study on a cylinder of $100 \unit{nm}$ length
and $25 \unit{nm}$ of diameter
where on $SideA$ the $z$-component of the velocity
is set equal to $0.1 \unit{\mu m s^{-1}}$,
on $SideB$ the Neumann condition 
$\underline{\underline{\sigma}}\mathbf{n}=2 \cdot 10^{-4} \mathbf{n} 
\unit{Pa}$ is applied while a no-stress condition 
is applied on the lateral surface.
The numerical solutions obtained using MP-FEMOS with the TH or HFB pairs
are indistinguishable from those computed by the commercial 
software COMSOL~\cite{COMSOL}.
\begin{figure}[htbp]
\centering
\subfigure[\label{fig:P1stabvel}]
{\includegraphics[width=0.48\columnwidth]{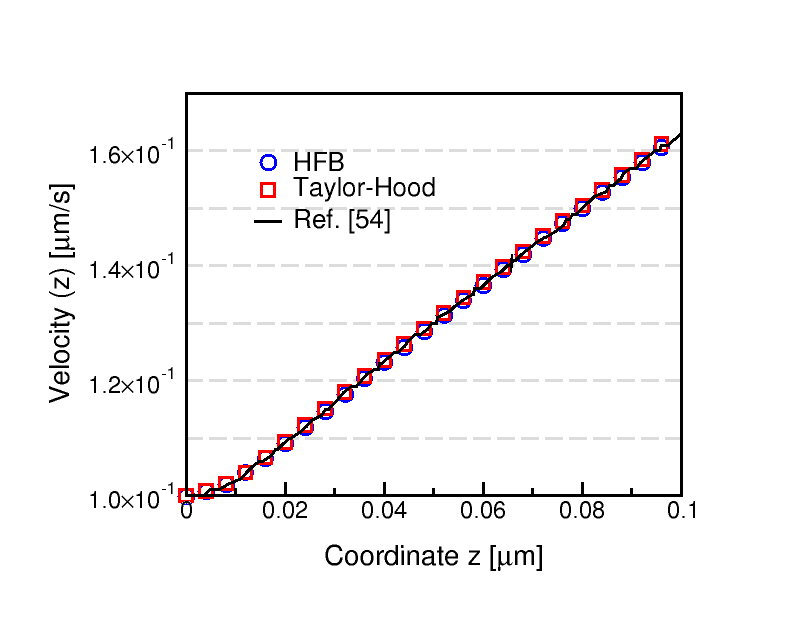}}
\subfigure[\label{fig:P1stabp}]
{\includegraphics[width=0.48\columnwidth]{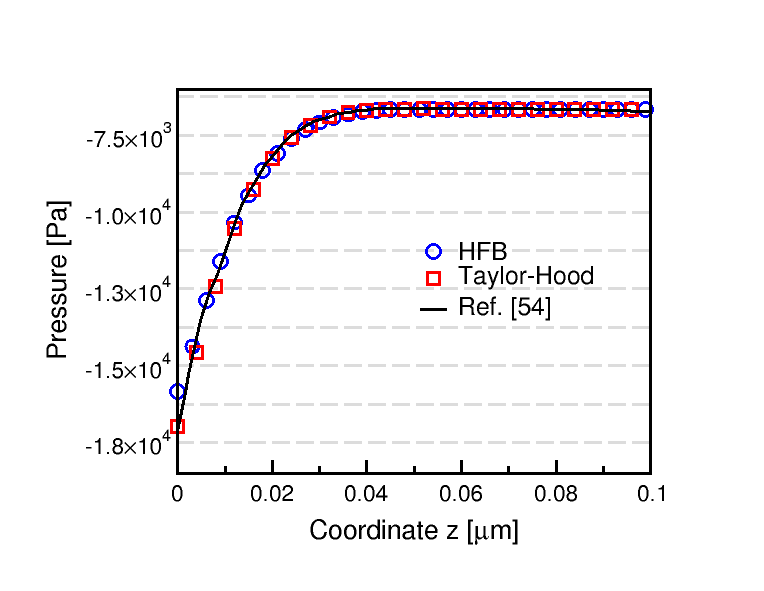} }
\caption{Comparison between the different discretization spaces for eq.~\eqref{eq:stpr}
and the result obtained using the commercial software COMSOL~\cite{COMSOL}: 
(a) velocity along $z$, (b) fluid pressure.}
\end{figure}

\subsection{Effects of a thermal gradient in the VE-PNP model}\label{sez2}

In the context of the numerical simulation of ion migration in phase change materials for electronics applications~\cite{NovielliIEDM2013} we have introduced the effect of a thermal gradient in the physical system
as a further driving force in the basic electrodiffusion model of ions~\eqref{eq:massPNP}.
In this section we aim to understand the role of such phenomenon 
in the context of a biological environment.
To verify this issue we perform a computational test in which we have 
deliberately chosen a non-physiological value of the temperature
located outside the cell ($SideB$ terminal).
The thermal gradient is consequent result of the following 
boundary conditions in~\eqref{eq:firstTE}:
\begin{align*}
&T = 293.75 \ \unit{K}& \ \ \ \mbox{on} \ \mbox{Bot} &\\ 
&T = 343.75 \ \unit{K}& \ \ \ \mbox{on} \ \mbox{Top} &\\
&T = 373.75 \ \unit{K}& \ \ \ \mbox{on} \ \mbox{Top} &\\ 
&\nabla T \cdot \mathbf{n} = 0 &  on \ \Gamma_N^{n_i}.&
\end{align*}
We named the previous conditions as $case 1$ and 
$case 2$ respectively where we have
neglected any heat generation phenomena in~\eqref{eq:Teq}.
The thermal profiles found inside the ionic channel after the simulation
are complicated by the fact that the Dirichlet boundary conditions
 are applied only on the channel terminals and not on the 
 whole $\Gamma$ as 
shown in Figs.~\ref{fig:temp3D} and~\ref{fig:temp1D}.
\begin{figure}[htbp]
\subfigure[\label{fig:temp3D}]
{\includegraphics[width=0.55\columnwidth]{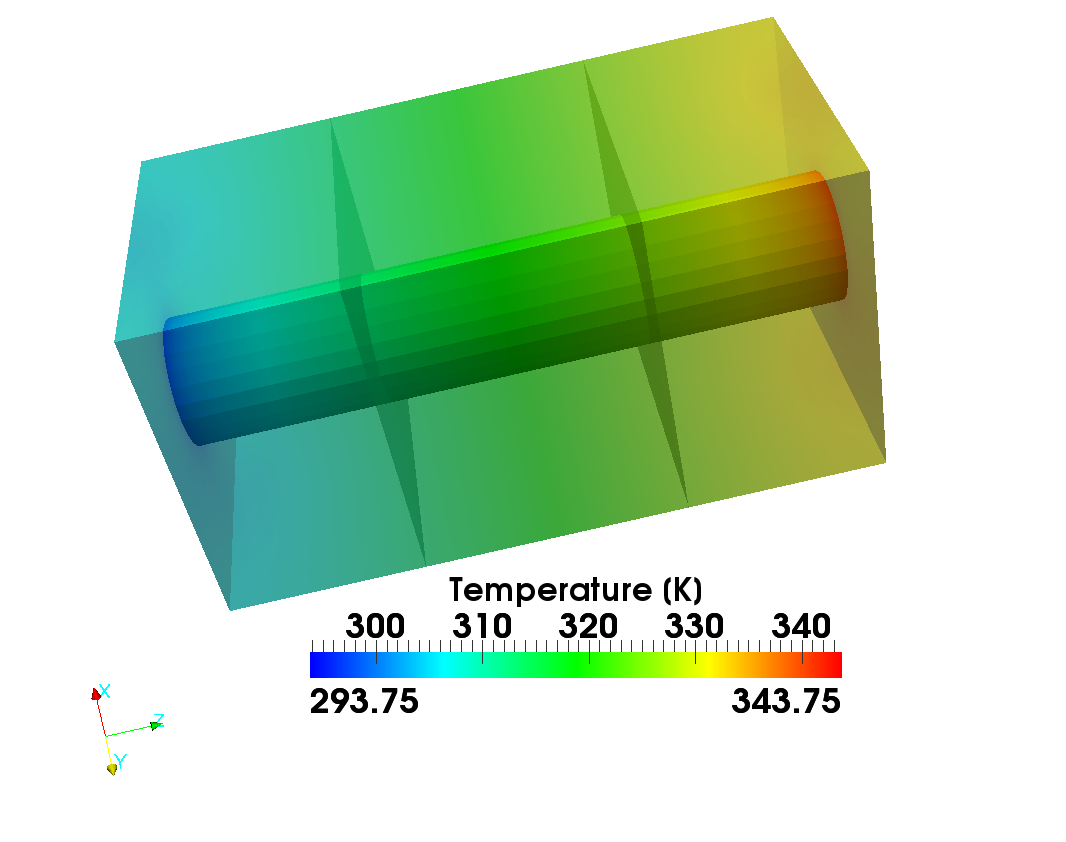}}
\subfigure[\label{fig:temp1D}]
{\includegraphics[width=0.5\columnwidth]{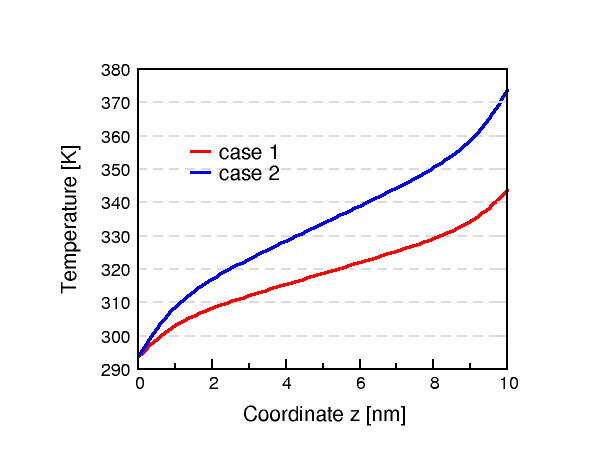} }
\caption{Channel thermal profiles resulting from~\eqref{eq:Teq}: (a) 
3D view of $case1$; (b) 1D-cut along the geometrical symmetry axis
for both $case1$ and $case2$.}
\end{figure}

Figs.~\ref{fig:Ktemp} and~\ref{fig:Natemp} show
the effect of the thermal gradient on K$^+$ and Na$^+$ 
concentration profiles: 
1D-cut are at the final time $T_f=50 \unit{ns}$
along the geometrical symmetry axis. Black color refers to
the constant temperature case $T=300 \unit{K}$, 
the red color to $case 1$ and the blue to $case 2$.
As the thermal gradient increases, the ionic species 
are drifted towards the $SideA$ surface, as prescribed by the 
negative sign in eq.~\eqref{eq:Teq}: Na$^+$ is more diffused
while K$^+$ is less.
\begin{figure}[htbp]
\subfigure[\label{fig:Ktemp}]
{\includegraphics[width=0.5\columnwidth]{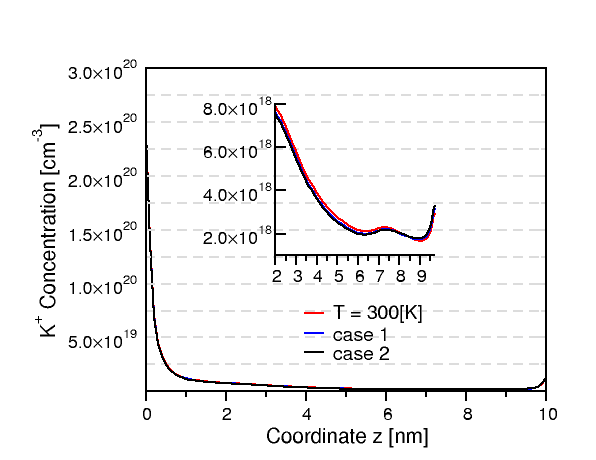}}
\subfigure[\label{fig:Natemp}]
{\includegraphics[width=0.5\columnwidth]{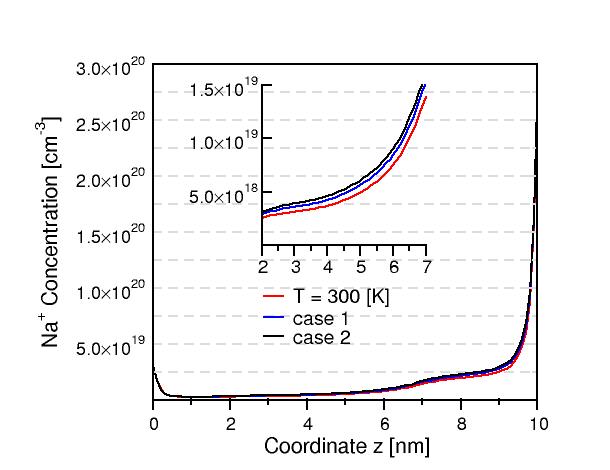} }
\caption{ K$^+$ (a) and Na$^+$ (b) concentration profiles: 
red color constant temperature at $T=300 \unit{K}$, 
blue color $case 1$ and black color $case 2$.}
\end{figure}

The electrostatic potential, shown in fig.~\ref{fig:pottemp},
is increased in the left
part of the channel due to the higher ion concentrations
in that region when the thermal gradient is not null.
For the same reason the influence of the 
ionic pressure onto the fluid is more pronounced as shown
in fig.~\ref{fig:veltemp}.  

\begin{figure}[htbp]
\subfigure[\label{fig:pottemp}]
{\includegraphics[width=0.5\columnwidth]{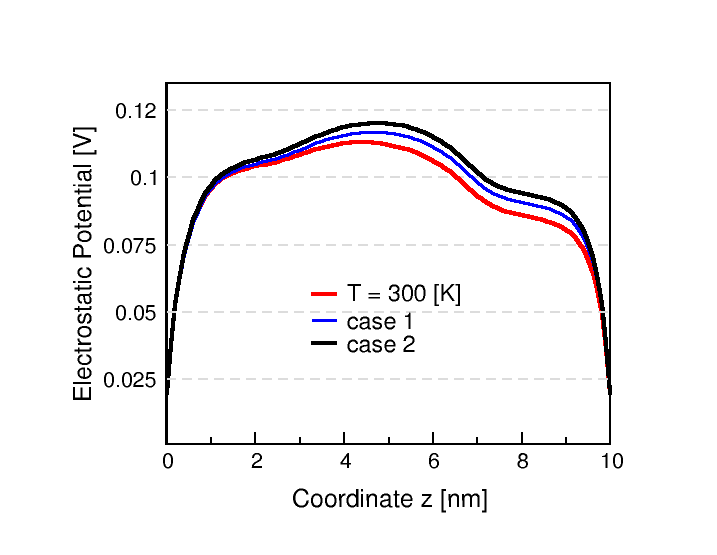}}
\subfigure[\label{fig:veltemp}]
{\includegraphics[width=0.5\columnwidth]{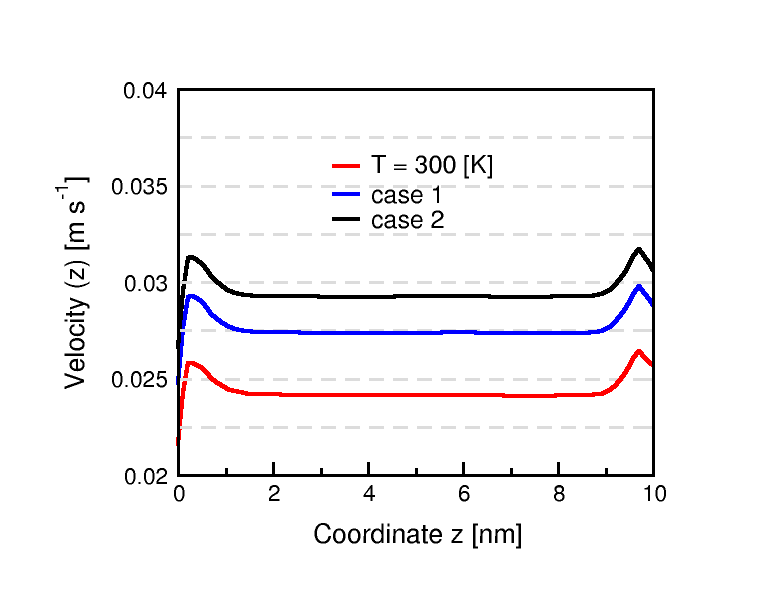} }
\caption{Electrostatic potential (a) and 
$z$-component of the fluid velocity (b) profiles: 
red color constant temperature at $T=300 \unit{K}$, 
blue color $case 1$ and black color $case 2$.}
\end{figure} 

From the computational analysis of this section we can conclude
that the contribution of a thermal driving force to ionic
flow is not relevant in biological systems unlike the case 
of semiconductor materials illustrated in~\cite{NovielliIEDM2013}.

%

\subsection{Numerical results for the VE-PNP model considering mechanical stress}\label{sez3}

In this section we demonstrate the effect on ion electrodiffusion 
due to the action of external mechanical forces able to modify the
size of the biological channel. This aspect of cellular biology 
is fundamental in many control mechanisms of a living system, for example,
heart beat, muscle activity or cerebrovascular autoregulation
(see~\cite{KeenerSneyd1998} for details and further reference).
As discussed in sec.~\ref{sec:algorithms}, the solution of the mechanical problem~\eqref{eq:MECeq} is used to determine the finite element mesh transformation that yields the new domain geometry
in which~\eqref{eq:pnpprob} and~\eqref{eq:stpr} are solved
with a split domain algorithm.
The following boundary conditions are applied to solve~\eqref{eq:MECeq} where
$\Gamma_N$ is the lateral surface of the whole parallelepiped,
$\Omega_1$, $\Omega_2$ and $\Omega_3$ 
are three separate regions with different mechanical properties 
as reported in Tab.~\ref{tab:PNPdata}.

\begin{subequations}
\begin{align}
&\vect{d_D} = [0, 0, 0] ^T \, \unit{\mu m} & 
\text{on } SideA\  \cup SideB \cup \Gamma_L \cup \Gamma_B \cup \Gamma_T 
\label{MECbc_a} \\
&\vect{f_{mec}} = [0, 0, 0] ^T \, \unit{N cm^{-3}} & 
\text{in }\Omega_1 \cup \Omega_2 \cup \Omega_3 \label{MECbc_b}\\
&\underline{\underline{\sigma_0}} = [\gamma,\gamma, 0, 0, 0, 0] ^T 
\, \unit{N cm^{-2}} & \text{in} \ \Omega_2 \label{MECbc_c}\\
&\gamma = -1\cdot 10^7  \label{MECbc_d}\\
&\gamma = -2\cdot 10^7. \label{MECbc_e}
\end{align}
\end{subequations}
We denote condition~\eqref{MECbc_d} as deformation 1 ($def.1$) 
while~\eqref{MECbc_e} as deformation 2 ($def.2$): the case 
with no deformation is indicated with $no-def.$.
An example of deformed mesh obtained after the solution 
of~\eqref{eq:MECeq} in the case of $def.1$ is 
shown in fig.~\ref{fig:domMEC}. The deformed mesh is 
obtained by simply stretching each element according to the 
computed displacement and avoiding the insertion of
any new vertex or further mesh refinement. 

Fig.~\ref{fig:onlyStokes} shows 
the impact of channel size modification on the 
fluid velocity profiles ($z$-component) (no-ionic pressure is considered). 
Results have been obtained by solving only~\eqref{eq:stpr} 
in the deformed domain, with 1D-cut 
taken along the geometrical symmetry axis.
As expected, the average value of the velocity decreases with 
the increase of deformation, 
but in the region where the channel is restricted
the electrolyte fluid velocity increases.
\begin{figure}[h!]
\centering
\includegraphics[width=0.75\columnwidth]{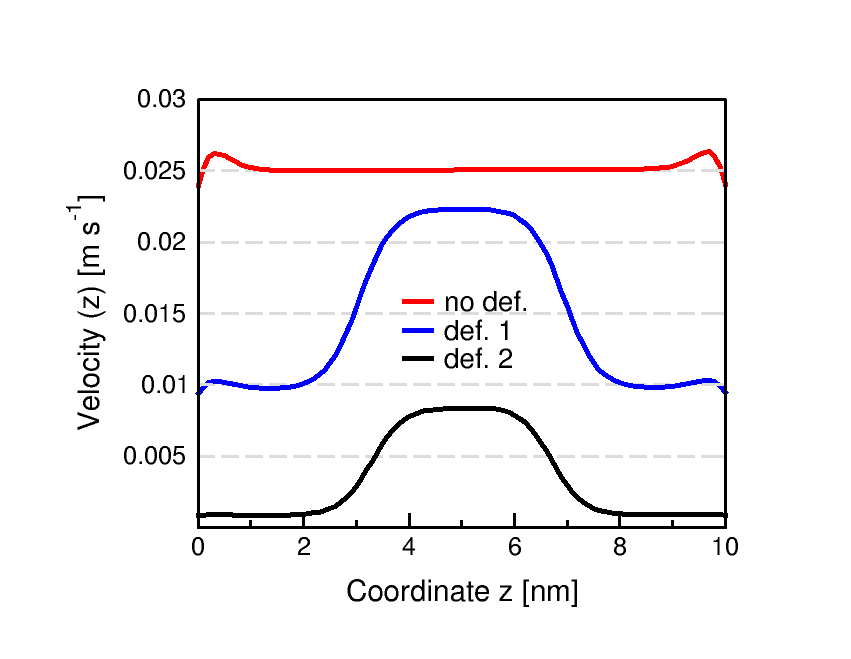}
\caption{Stokes equation in stationary conditions: 
1D-cut along the geometrical symmetry axis of
the $z$-component of the fluid 
velocity before and after mesh deformation.}
\label{fig:onlyStokes}
\end{figure}

\begin{figure}[h!]
\subfigure[\label{fig:MECpot1D}]
{\includegraphics[width=0.48\columnwidth]{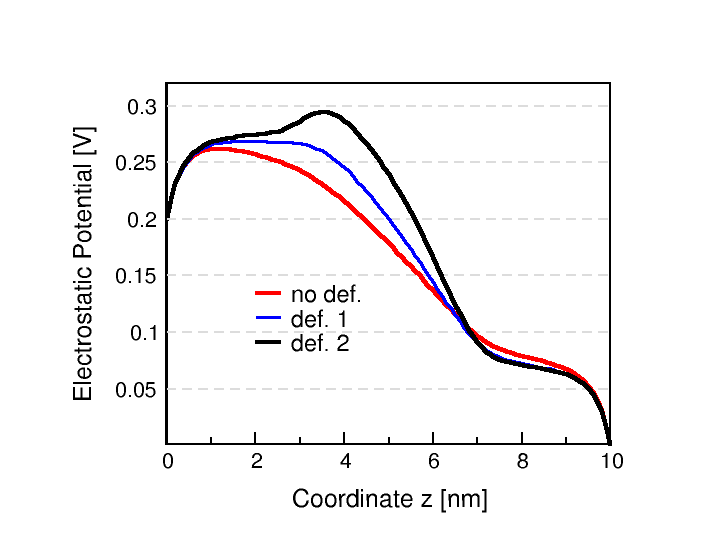} }
\subfigure[\label{fig:MECE1D}]
{\includegraphics[width=0.48\columnwidth]{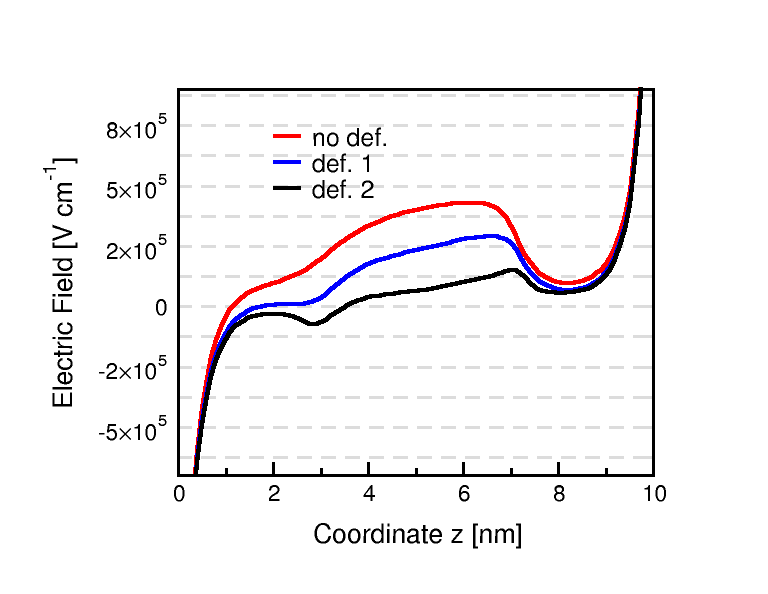}}
\caption{VE-PNP model in presence 
of mechanical deformation:
1D-cut along the geometrical symmetry axis
of the electrostatic potential (a) and of the electric field (b)
at the final time $T_f=50 \unit{ns}$.}
\end{figure}

Before describing the ion concentration profiles 
computed by the VE-PNP model in presence of mechanical forces
it is important to look at the profiles of the driving 
forces acting on ion electrodiffusion.
Figs.~\ref{fig:MECpot1D} and~\ref{fig:MECE1D} illustrate
the electrostatic potential and electric field
profiles with and without deformation.
Increasing the deformation the electric field strength 
in the channel center is quite reduced while the region
close to $SideA$ where the electric field is negative is extended towards
the channel center. In all cases the electric field drop
close to $SideB$ corresponds to the almost flat profile in the 
electrostatic potential.
\begin{figure}[h!]
\centering
\includegraphics[width=0.5\columnwidth]{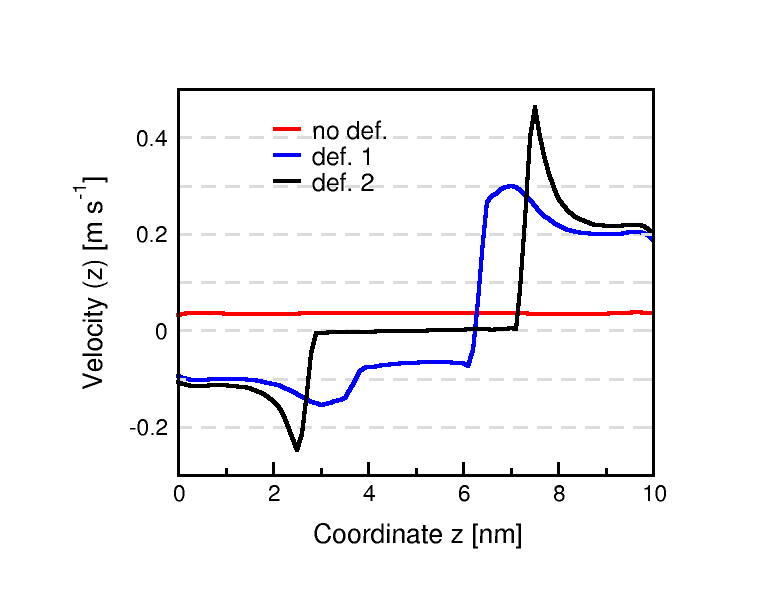}
\caption{VE-PNP model in presence 
of the mechanical deformation:
1D-cut along the geometrical symmetry axis of the $z$-component of the 
fluid velocity at the final time $T_f=50 \unit{ns}$
clearly highlight the effect of ionic pressure on
the fluid.}
\label{fig:MECvel1D}
\end{figure}

The effects of ionic pressure on the fluid velocity
are shown in fig.~\ref{fig:MECvel1D}: not only is the velocity
not constant in the channel but also it undergoes a change of sign
(inversion of electrolyte flow).
From the cell-inside to the two third of the channel the velocity
is directed towards the cell while in the final part of the channel 
it moves in the opposite direction. Moreover increasing
deformation, a reduction of the velocity at the center of the 
channel is found due to the reduction of the electric field.

Fig.~\ref{fig:MECconc1D} illustrates the K$^+$ and Na$^+$ concentration
profiles.
For both species the main variation with respect to the unperturbed 
channel configuration is obtained in the center portion of the channel
but with a remarkable difference. 

In the case of K$^+$, Fick's diffusion is towards the end of the channel
as the electric field and velocity in the case of
normal size channel. This produces the K$^+$ accumulation 
at the end of the channel: the Dirichlet boundary condition
causes the reduction of the concentration at $SideB$.
With the deformation, fluid velocity becomes negative and is
directed towards the cell so that accumulation no longer occurs.

In the case of Na$^+$, Fick's diffusion is directed
towards the cell while the velocity is towards the
outside with an increased strength under deformation: 
this is the reason why the deformed 
channel has a less diffused profile.
\begin{figure}[h!]
\centering
\subfigure[\label{fig:kdef}]
{\includegraphics[width=0.475\columnwidth]{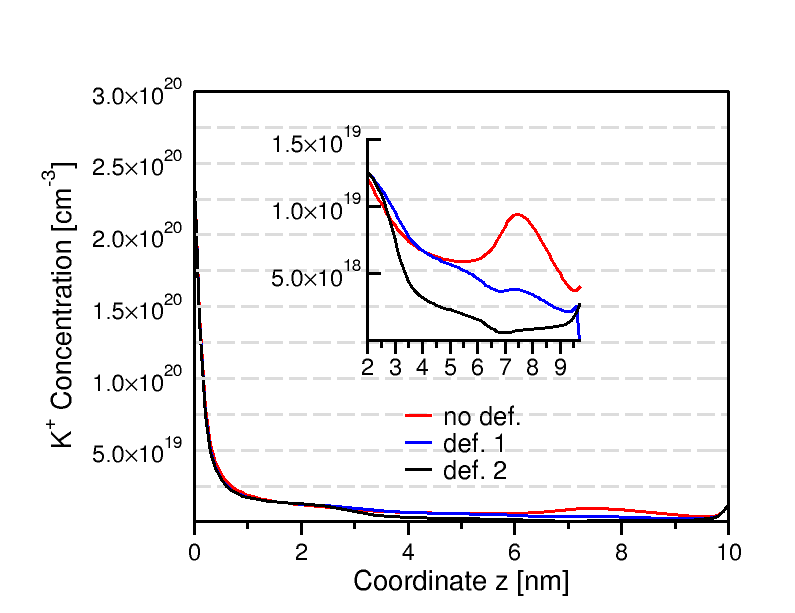}}
\subfigure[\label{fig:Nadef}]
{\includegraphics[width=0.475\columnwidth]{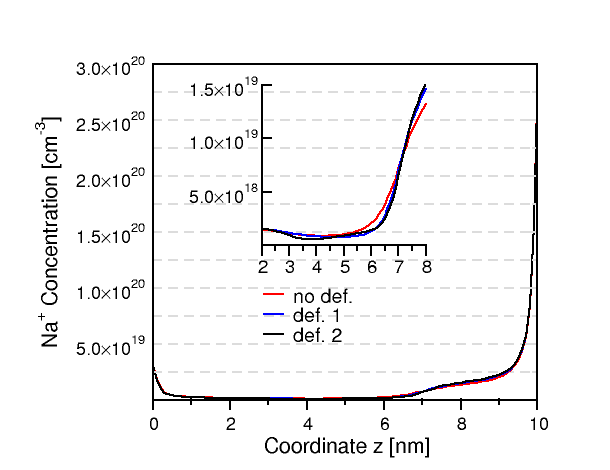} }
\caption{VE-PNP model in presence of mechanical deformation:
1D-cut along the geometrical symmetry axis of the ionic concentration
at the final time $T_f=50 \unit{ns}$: (a) K$^+$, (b) Na$^+$.}
\label{fig:MECconc1D}
\end{figure}

The 3D plots of fig.~\ref{fig:MECdispl3D} show how
the deformation is along the normal direction of the channel axis
and it increases with the 
deformation strength $\gamma$ in~\eqref{MECbc_d} and~\eqref{MECbc_e}.
As a consequence, channel diameter is further and further reduced.
Fig.~\ref{fig:MECpot3D} shows how the potential profile
of ~\ref{fig:MECpot1D} in a 3D view is uniform across the
channel with no peaks induced by the possible
accumulation of ionic charge at the channel surfaces.
Fig.~\ref{fig:MECvel3D} shows the 3D view of
the $z$-component of the fluid velocity: the profile
across the cell is nonuniform  also in the undeformed
case due to the adherence condition: the more
the deformation, the more the nonuniformity of velocity distribution.
Fig.~\ref{fig:MECpress3D} illustrates the 3D profile
of Na$^+$ ionic concentration. As in the case of electrolyte fluid,
the increased channel deformation significantly affects the nonuniformity
of the 3D distribution of the sodium ion inside the channel.
This effect might be very important in regulating the electrostatic 
coupling between the mobile charge distribution in the channel 
and the permanent charge that is usually trapped inside the lipid
membrane bilayer surrounding the channel. This biophysical aspect
of the problem has been investigated 
by Prof.~W.~Nonner and coworkers in~\cite{Malasics2009,Boda2009}
and will be object of a future step of our research.
\begin{figure}[htbp]
\centering
\subfigure[\label{fig:MECdispl3D}]
{\includegraphics[width=0.35\columnwidth]{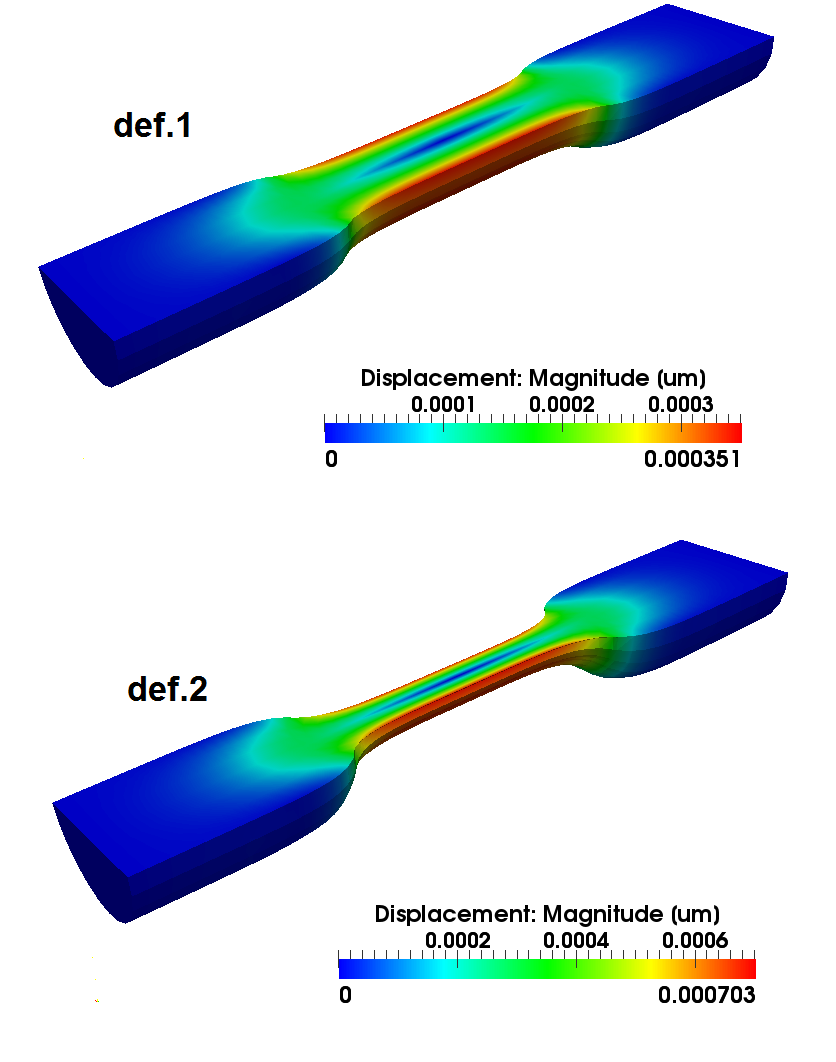}}
\subfigure[\label{fig:MECpot3D}]
{\includegraphics[width=0.3\columnwidth]{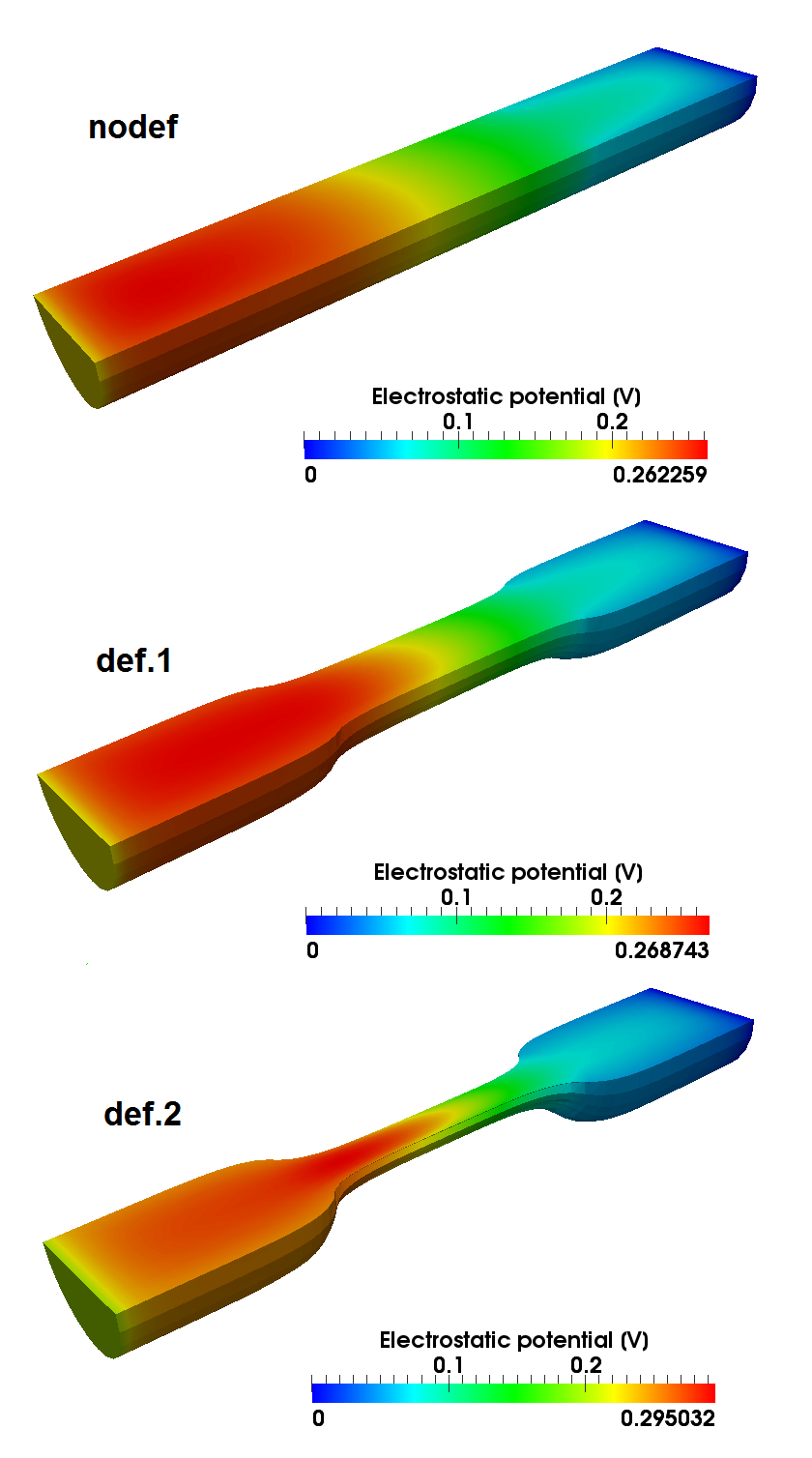} }
\caption{VE-PNP model in presence 
of mechanical deformation:
3D interior view of displacement (a) and electrostatic potential
(b) at the final time $T_f=50 \unit{ns}$.}
\end{figure}

%

\begin{figure}[htbp]
\centering
\subfigure[\label{fig:MECvel3D}]
{\includegraphics[width=0.35\columnwidth]{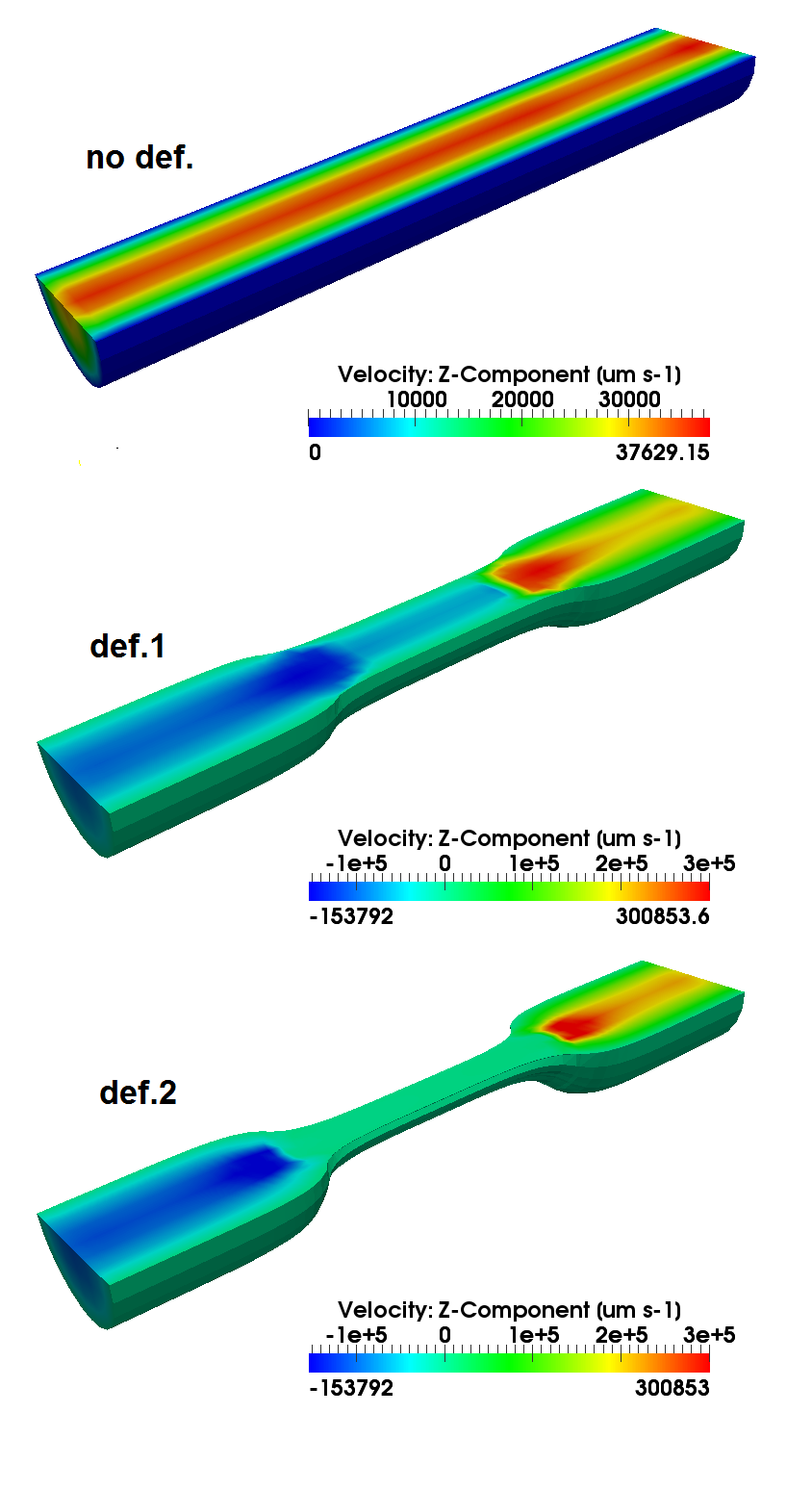}}
\subfigure[\label{fig:MECpress3D}]
{\includegraphics[width=0.35\columnwidth]{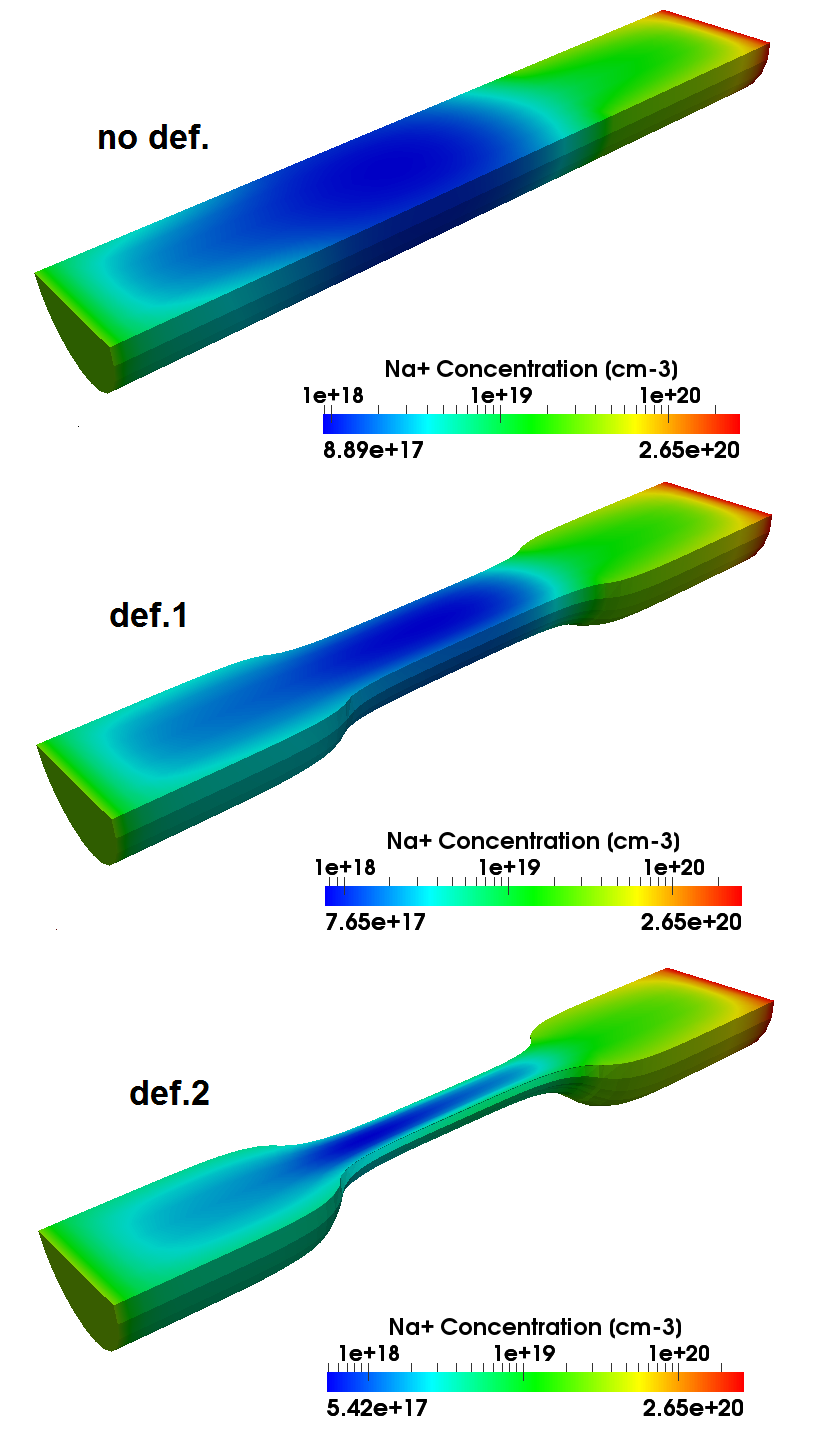} }
\caption{VE-PNP model in presence 
of mechanical deformation:
3D interior view of the $z$-component of fluid velocity (a) 
and of Na$^+$ concentration (b) at the final time $T_f=50\unit{ns}$.}
\end{figure}

%

%

\section{Conclusions and Future Perspectives}\label{sec:conclusions}

In this article we have carried out a 
three-dimensional modeling and simulation, 
in three spatial dimensions and in time dependent conditions, 
of biological ion channels using a continuum-based approach. 
The proposed model is a multi-physics formulation that is capable to 
combine, to the best of our knowledge for the first time, 
ion electrodiffusion, channel fluid motion, thermal self-heating
and mechanical deformation. The full self-consistency is intended
for further study. Here, the thermal and mechanical effects 
are unidirectional. 

The resulting mathematical picture is a system of nonlinearly
coupled partial differential equations in conservation form 
that describes the fundamental principles of ion charge and momentum 
balance, heat flow balance, electrolyte mass and momentum balance
and mechanical balance in the nanochannel.

A functional iteration that decouples the various differential 
subblocks of the system is introduced with the purpose of 
facilitating the computational implementation. Then, each resulting
subproblem is discretized using the Galerkin Finite Element Method.

The validation of the proposed computational model is carried out 
by simulating a realistic cylindrical voltage operated ion nanochannel 
where K$^+$ and Na$^+$ ions are simultaneously flowing. 
Three sets of numerical experiments are performed and thoroughly
discussed. We first investigate the coupling between 
electrochemical and fluid-dynamical effects. Then, we 
enrich the modeling picture by investigating the influence of a thermal gradient. Finally, we add a mechanical stress 
responsible of channel deformation and investigate its effect 
on the functional response of the channel. 
Results show
that fluid and thermal fields have no influence 
in absence of mechanical deformation whereas ion distributions and 
channel functional response are significantly modified 
if mechanical stress is included in the model. These predictions agree with 
biophysical conjectures on the importance of protein conformation 
in the modulation of channel electrochemical properties 
developed by Prof.~W.~Nonner and coworkers in~\cite{Malasics2009,Boda2009}.

Future extensions to overcome present model/methods limitations include, 
being not restricted to:
\begin{itemize}
\item the coupling between stress and deformation fields
and the PNP system to self-consistently 
calculate the deformation of channel wall;
\item a biophysically sound characterization of the heat production term 
$Q_{heat}$ in the thermal block to account the effect of 
chemical reactions at the entrance of channel mouth between ions
and, for example, externally supplied drugs;
\item the electrostatic 
coupling between the mobile charge distribution in the channel 
and the permanent charge trapped inside the lipid
membrane bilayer;
\item the extension and analysis of the EAFE formulation 
to account for electrolyte velocity in the thermo-electrical 
drift term in the discretization of the T-VE-PNP model.
\end{itemize}

\bibliographystyle{plain} 
\bibliography{biblio}


\end{document}